\documentclass[11pt,english]{article}
\usepackage[T1]{fontenc}
\usepackage[latin9]{inputenc}
\usepackage{geometry}
\usepackage{verbatim}
\usepackage{float}
\usepackage{units}
\usepackage{amsmath}
\usepackage{amsthm}
\usepackage{amssymb}
\usepackage{graphicx}
\usepackage{setspace}
\makeatletter

\floatstyle{ruled}
\newfloat{algorithm}{tbp}{loa}
\providecommand{\algorithmname}{Algorithm}
\floatname{algorithm}{\protect\algorithmname}

\theoremstyle{plain}
\newtheorem{thm}{\protect\theoremname}[section]
\theoremstyle{definition}
\newtheorem{example}[thm]{\protect\examplename}
\theoremstyle{definition}
\newtheorem{defn}[thm]{\protect\definitionname}
\theoremstyle{plain}
\newtheorem{lem}[thm]{\protect\lemmaname}
\ifx\proof\undefined
\newenvironment{proof}[1][\protect\proofname]{\par
\normalfont\topsep6\p@\@plus6\p@\relax
\trivlist
\itemindent\parindent
\item[\hskip\labelsep\scshape #1]\ignorespaces
}{%
\endtrivlist\@endpefalse
}
\providecommand{\proofname}{Proof}
\fi

\makeatother

\usepackage{babel}
\providecommand{\definitionname}{Definition}
\providecommand{\examplename}{Example}
\providecommand{\lemmaname}{Lemma}
\providecommand{\theoremname}{Theorem}

\begin{document}

\title{$\theta$-parareal schemes }

\author{Gil Ariel\thanks{Bar-Ilan University, Ramat Gan, Israel}, Hieu Nguyen\thanks{The University of Texas at Austin, USA}
and Richard Tsai\thanks{The University of Texas at Austin, USA and KTH Royal Institute of
Technology, Sweden }}
\maketitle
\begin{abstract}
A weighted version of the parareal method for parallel-in-time computation
of time dependent problems is presented. Linear stability analysis
for a scalar weighing strategy shows that the new scheme may enjoy
favorable stability properties with marginal reduction in accuracy
at worse. More complicated matrix-valued weights are analyzed and
applied in numerical examples. The weights are optimized using information
from past iterations, providing a systematic framework for using the
parareal iterations as an approach to multiscale coupling. The advantage
of the method is demonstrated using numerical examples, including
some well-studied nonlinear Hamiltonian systems.
\end{abstract}

\section{Introduction}

Parallelization of computation for spatial domain, such as the standard
domain decomposition methods, has been extensively developed and successfully
applied to many important applications. Due to causality, parallel-in-time
computations have not been as successful as parallel computations
in space. However, numerical simulations will not benefit from available
exa-scale computing power unless parallelization-in-time can be performed.
Despite recent advances, the presence of strong causalities in the
sense that local perturbations are not damped out by the system's
dissipation, e.g. in hyperbolic problems and fast oscillations in
the solutions, typically hinders the efficiency of such types of algorithms.
For example, approaches involving shooting and Newton's solvers may
become virtually unusable. It is widely recognized that robust and
convergent numerical computation using such parallel-in-time algorithms
still remains a main challenge. 

Several attempts for designing time-parallel algorithms for evolutionary
problems have been proposed. The common idea is to decompose the time
domain of interest into several subintervals. In each subinterval,
the given equation is solved in parallel with time-boundary conditions
given at one or both ends of each subinterval. The time-boundary conditions
are coupled via some specific algorithms, typically of iterative nature.
With multiple shooting methods, see e.g. \cite{keller1976numerical,kiehl1994parallel},
one solves a two-point boundary value problem in each subinterval,
and uses a Newton's iterations to couple all the boundary conditions
together. Particularly for problems with oscillations, Newton iterations
may not converge. A different approach, termed parareal, was proposed
by Lions, Maday and Turinici in \cite{parareal-LMT01}. In the parareal
framework, one solves an initial value problem in each subinterval
with a high-accuracy ``fine solver'', starting from the time-boundary
conditions computed by a stable ``coarse'' solver. The coarse solution
at the boundary of each subintervals is ``corrected'' iteratively
by adding back the difference between the fine and the coarse solutions
computed in the previous iteration. The standard parareal scheme was
found to work quite well for dissipative problems. Loosely speaking,
the parareal iterations typically converge quite well to the desired
solution as long as it is stable.

In order to widen the range of applicability of parareal and increase
its stability, several methods, combining parareal with other approaches,
have been suggested. For example, Minion \cite{minion2011hybrid}
proposes a ``deferred spectral correction'' scheme. Farhat and Chandesris
\cite{Farhat2003} add a Newton-type iteration to reduce the jumps
between the fine and coarse solutions. Gander et. al. \cite{GanderPetcu2008}
analyze the Krylov subspace approach of \cite{Farhat2003} for linear
Ordinary Differential Equations (ODEs). The main idea is to use past
iterations to form a subspace that can improve the coarse integrator.
Although this method is applicable for low dimensional systems, it
would become insufficient for high dimensional problems due to difficulties
in orthogonalization in a large subspace. In \cite{GuettelGander2013},
the authors manipulate the principle of superposition in linear ODEs
to decouple inhomogeneous equations. Applying fast exponential integrators
that are highly efficient for the homogenized part, the method is
applicable to high-dimensional linear problems. Applications of parareal
methods to Hamiltonian dynamics have been analyzed in \cite{GanderHairer14}.
Additional approaches applying symplectic integrators with applications
to molecular dynamics include \cite{Bal-LNCSE08,Jimenez2011}. Dai
et. al. \cite{Dai2013} proposed a symmetrized parareal version coupled
with projections to the constant energy manifold.

In \cite{parareal-Legoll13}, a multiscale parareal scheme is proposed
for dynamical systems possessing fast dissipative dynamics. It is
found that the fast dissipative dynamics deteriorate the convergence
of the ``standard'' parareal scheme, and suitable projections of
the fast variables may improve the convergence property of such types
of systems. In \cite{AKT-SISC-parareal}, a parareal like multiscale
coupling schemes are proposed for highly oscillatory dynamical systems.
In that work, the coarse integrator in the standard parareal scheme
is replaced by a multiscale integrator that solves an effective system
derived from the given highly oscillatory one. The coarse solutions
are enhanced by an ``alignment'' process that uses the current fine
solutions. The idea of aligning the fine and coarse solutions and
propagating corrections on the coarse grid is also similar to the
correction method proposed in \cite{Farhat2003}. Both of these approaches
may be considered a special case of the general weighing scheme proposed
in this paper. Several works have addressed the applicability of parareal
methods to hyperbolic equations. It has been shown, that hyperbolic
problems pose stability issues for parareal iterations, especially
with large steps \cite{dai2013stable,ruprecht2012explicit,Farhat2003}.
Applications include structural models \cite{Farhat2003}, acoustic
advection problems \cite{ruprecht2017wave} and Partial Differential
Equations (PDEs) with highly oscillatory forcing \cite{parareal-HS-PDEs13}.

In this paper, we propose time-parallel algorithms motivated by the
parareal methods of \cite{parareal-LMT01}, due to its simple, derivative
free, iterative structure. The main goal is to enhance the stability
of the parareal iterations by taking a weighted linear combination
of the previous and current iterations. The new method is termed $\theta$-parareal
due to its formal resemblance to the known $\theta$-schemes for discretizing
time dependent partial differential equations. Particular emphasis
is given to oscillatory dynamical systems with essentially no dissipation.
Furthermore, we provide a systematic approach for coupling computations
involving different but in some sense ``nearby'' time dependent
problems. 

The paper is organized as follows. Section 2 presents our main approach
and analyzes some of its important properties. Section 3 presents
numerical examples. We conclude in section 4.

\section{$\theta$-parareal}

Consider ODEs of the form,
\[
\frac{d}{dt}u=f(u),\:u(0)=u_{0}.
\]
We are interested in a numerical approximation of the solution in
a bounded time segment $[0,T]$. Throughout the paper it is assumed
that solutions exist in $[0,T]$ and are sufficiently smooth.

Let $u_{n}^{(k)}\in\mathbb{C}^{d}$ denote the solution computed by
the parareal schemes at iteration $k$ and time $t_{n}=nH.$ Let $F_{H}$
and $C_{H}$ denote the numerical propagators used as the fine (high
accuracy but expansive) and coarse (low accuracy but cheap) integrators
up to time $H$. The parareal scheme proposed in \cite{parareal-LMT01}
is defined by the following simple iterations,

\begin{equation}
u_{n+1}^{(k+1)}=C_{H}u_{n}^{(k+1)}+\left(F_{H}u_{n}^{(k)}-C_{H}u_{n}^{(k)}\right),\,\,\,n,k=0,1,2,\dots,\label{eq:original-parareal}
\end{equation}
with the initial conditions 
\begin{equation}
u_{0}^{(k)}=u_{0},\,\,\,k=0,1,2,\dots.\label{eq:IC}
\end{equation}
 The first (zero) iteration is taken as
\[
u_{n+1}^{(0)}=C_{H}u_{n}^{(0)},\,\,\,n=0,1,2,\dots.
\]
We shall refer to (\ref{eq:original-parareal}) as the standard parareal
scheme. 

Consider a weighted version of the parareal update,
\[
u_{n+1}^{(k+1)}=\theta C_{H}u_{n}^{(k+1)}+(1-\theta)C_{H}u_{n}^{(k)}+\left(F_{H}u_{n}^{(k)}-C_{H}u_{n}^{(k)}\right),
\]
 leading to a more symmetric form,
\begin{equation}
u_{n+1}^{(k+1)}=\theta C_{H}u_{n}^{(k+1)}+\left(F_{H}u_{n}^{(k)}-\theta C_{H}u_{n}^{(k)}\right),\label{eq:proposed-parareal}
\end{equation}
 where $\theta$ are mappings from $\mathbb{C}^{d}$ to $\mathbb{C}^{d}$,
which may depend on $k$ and $n$. In the general case, the weights
$\theta$ will be denoted $\theta_{n}^{(k)}$, i.e.,

\begin{equation}
u_{n+1}^{(k+1)}=\theta_{n+1}^{(k+1)}C_{H}u_{n}^{(k+1)}+\left(F_{H}u_{n}^{(k)}-\theta_{n+1}^{(k+1)}C_{H}u_{n}^{(k)}\right).\label{eq:Ansatz-for-optimizing-theta}
\end{equation}
We start with the simplest case where $\theta$ is a real number,
then a complex number and finally linear operators. We note that the
method can still be parallelized as the initial condition for the
fine integrator only depends on the previous iteration. We view $\theta C_{H}$
as a new coarse integrator, and investigate in what (simple) ways
$\theta$ can enhance stability and accuracy of the original parareal
($\theta\equiv1$). 

We shall first show that the new schemes preserve the ``exact causal
property'' as the original parareal scheme, i.e., that, given in
$H$, the method will always converge to the fine solutions $(F_{H})^{n}u_{0}$
after $T/H$ iterations. Indeed, we notice that if $u_{n}^{(k)}=u_{n}^{(k+1)},$
then the recurrence relations in (\ref{eq:original-parareal}) or
(\ref{eq:proposed-parareal}) reduce to advancing from $t_{n}$ to
$t_{n}+H$ using the fine scale integrator, i.e., $u_{n}^{(k)}=(F_{H})^{n}u_{0}$
is a fixed point. More precisely, given the initial condition (\ref{eq:IC})
we see that 
\[
u_{1}^{(1)}=\theta_{1}^{(1)}C_{H}u_{0}+(F_{H}u_{0}-\theta_{1}^{(1)}C_{H}u_{0})=F_{H}u_{0},
\]
 and 
\[
u_{1}^{(k)}=\theta_{1}^{(k)}C_{H}u_{0}^{(k)}+(F_{H}u_{0}^{(k-1)}-\theta_{1}^{(k)}C_{H}u_{0}^{(k-1)})=F_{H}u_{0},k=1,2,\dots.
\]
By induction, 
\[
u_{j}^{(k+1)}=(F_{H})^{j}u_{0},\,\,\,j\le k,
\]
which implies that,
\begin{align*}
u_{k+1}^{(k+1)} & =\theta_{k+1}^{(k+1)}C_{H}u_{k}^{(k+1)}+(F_{H}u_{k}^{(k)}-\theta_{k+1}^{(k+1)}C_{H}u_{k}^{(k)})=(F_{H})^{k+1}u_{0},\,\,\,k=0,1,2,\dots.
\end{align*}
Hence, we have the following exact causality property:
\begin{thm}
\label{prop:exact-causal}Let $u_{n}^{(k)}$ solve (\ref{eq:original-parareal})
and (\ref{eq:IC}). Then,

\[
u_{n}^{(k)}=(F_{H})^{n}u_{0},\,\,\,\,\forall k\ge n.
\]
\end{thm}
The surprising thing about this result is that it holds even if the
effective coarse integrator $\theta C_{H}$ is not consistent with
the ODE.

We now consider a simple case in which $\theta,$ $C_{H}$ and $F_{H}$
are linear operators, independent of $n$ and $k$. In order to study
the stability and convergence of the $\theta$-scheme, let $v_{n+1}^{(k)}$
denote the correction term $F_{H}u_{n}^{(k)}-\theta C_{H}u_{n}^{(k)}$.
Then $\theta$-parareal can be written as,

\begin{eqnarray*}
u_{n+1}^{k+1} & = & \theta C_{H}u_{n}^{(k+1)}+v_{n+1}^{(k)}\\
 & = & (\theta C_{H}\circ\theta C_{H})u_{n-1}^{(k+1)}+\theta C_{H}v_{n}^{(k)}+v_{n+1}^{(k)}\\
 & \vdots\\
 & = & (\theta C_{H})^{n+1}u_{0}^{(k+1)}+\sum_{j=1}^{n+1}\left(\prod_{i=j+1}^{n+1}\theta C_{H}\right)v_{j}^{(k)}.
\end{eqnarray*}
 Here we use the notation,
\[
(\theta C_{H})^{\ell}u=(\prod_{j=1}^{\ell}\theta C_{H})u=\underbrace{\theta C_{H}\circ\theta C_{H}\cdots\theta C_{H}}_{\ell\text{ times}}u.
\]
For a fixed $k$, the stability of the time marching is determined
by $(\theta C_{H})^{n}$. Having a stable coarse solver is crucial
in stabilizing the parareal solution because the correction is often
small up to the order of accuracy. However, it will be interesting
to consider examples in which, introducing the factor $\theta$ can
stabilize iterations. Suppose one runs the parareal scheme in a time
interval consisting of $N$ coarse sub-intervals. Define, $U^{(k)}:=(u_{0}^{(k)},u_{1}^{(k)},\cdots,u_{N}^{(k)})^{T},$
$I_{0}=\left(u_{0},0,\ldots,0\right)^{T}$, and
\[
A=\left(\begin{array}{ccccc}
I & 0 & \ldots & 0 & 0\\
-\theta C_{H} & I & \ldots & 0 & 0\\
0 & -\theta C_{H} & \ddots & 0 & 0\\
\vdots & \vdots & \ddots & \vdots & \vdots\\
0 & 0 & \ldots & -\theta C_{H} & I
\end{array}\right),B=\left(\begin{array}{ccccc}
0 & 0 & \ldots & 0 & 0\\
F_{H}-\theta C_{H} & 0 & \ldots & 0 & 0\\
0 & F_{H}-\theta C_{H} & \ddots & 0 & 0\\
\vdots & \vdots & \ddots & \vdots & \vdots\\
0 & 0 & \ldots & F-\theta C_{H} & 0
\end{array}\right).
\]
Then, the $\theta$-parareal iteration can be written in matrix form
as, 
\[
AU^{(k+1)}=BU^{(k)}+I_{0},
\]
with initial condition $U^{(0)}=(u_{0},C_{H}u_{0},\cdots,(C_{H})^{N}u_{0}).$
Thus, we obtain an explicit expression of $U^{(k)}$,
\[
U^{(k+1)}=A^{-1}BU^{(k)}+A^{-1}I_{0}.
\]
 We readily see that $U^{*}=(u_{0},F_{H}u_{0},\cdots,(F_{H})^{N}u_{0})$
is a fixed point, $U^{*}=A^{-1}BU^{*}+A^{-1}I_{0}$. Denoting the
signed error $E_{n}^{(k)}=U^{(k)}-U^{*},$ it is given by,
\[
E^{(k+1)}=A^{-1}BE^{(k)}=\left(\begin{array}{cccccc}
0\\
I & 0\\
\theta C_{H} & I & 0\\
(\theta C_{H})^{2} & \ddots & \ddots & \ddots\\
\vdots & \ddots & C_{H} & I & 0\\
(\theta C_{H})^{N-1} & \cdots & (\theta C_{H})^{2} & \theta C_{H} & I & 0
\end{array}\right)(F_{H}-\theta C_{H})E^{(k)},
\]
where $E^{(0)}=(0,(F_{H}-C_{H})u_{0},(F_{H}^{2}-C_{H}^{2})u_{0},\cdots,(F_{H}^{N}-C_{H}^{N})u_{0})^{T}$.
We find the following theorem.
\begin{thm}
Let\textbf{ $e_{n}^{(k)}:=u_{n}^{(k)}-(F_{H})^{n}u_{0},$} $n=0,1,\cdots,N,$
and $k\le n.$ The following inequality holds,
\begin{equation}
|e_{n}^{(k+1)}|\le\|F_{H}-\theta C_{H}\|_{\infty}\sum_{j=0}^{n-k-2}\|\theta C_{H}\|_{\infty}^{j}\,|e_{n}^{(k)}|.\label{eq:Thm22}
\end{equation}
\end{thm}
Note that this estimate can be derived from Theorem 4.5 in \cite{gander2007analysis}
by formally replacing $C_{H}$ with $\theta C_{H}$ and $F_{H}$ with
the exact solution operator that advances the solution by a time length
$H$. Nonetheless, the difference is important for the discussion
below and provides insight into how a parareal method would perform,
depending on the stability of $F_{H}$, $C_{H}$, $N$, and the accuracy
of $C_{H}$. We would like to see under what conditions the parareal
iterations decrease the errors. This translates to finding conditions
that render the amplification factor $Q_{n,k}:=|F_{H}-\theta C_{H}|\sum_{j=0}^{n-k-2}|\theta C_{H}|^{j}<1.$
We immediately see that a deciding factor is whether $\sum_{j=0}^{n-k-2}|\theta C_{H}|^{j}$
is uniformly bounded in $n$. 
\begin{thm}
(nonlinear variable coefficient case) Let\textbf{ $e_{n}^{(k)}:=u_{n}^{(k)}-(F_{H})^{n}u_{0},$}
$n=0,1,\cdots,N,$ and $k\le n.$ Then\textbf{ $u_{n}^{(k)}$ }is
given by,
\end{thm}
\[
u_{n+1}^{k+1}=\left[\prod_{j=1}^{n+1}\theta_{j}^{(k)}C_{H}\right]u_{0}^{k+1}+\sum_{j=1}^{n+1}\left[\prod_{i=j+1}^{n+1}\theta_{i}^{(k)}C_{H}\right](F_{H}u_{j-1}^{k}-\theta_{j-1}^{(k)}C_{H}u_{j-1}^{k}).
\]
\textit{Furthermore, the following inequality holds,}
\begin{equation}
|e_{n}^{(k+1)}|\le\|F_{H}-\theta C_{H}\|_{\infty}\sum_{j=0}^{n-k-2}\|\theta C_{H}\|_{\infty}^{j}\,|e_{n}^{(k)}|.\label{eq:Thm23}
\end{equation}

The proof is similar to the linear case. We see that the solution
is composed of coarse solution and a series of propagating correction. 

\subsection{Linear theory }

We consider a diagonalizable linear system of first order differential
equations,
\begin{equation}
\frac{d}{dt}U=AU,\,\,\,U(t)\in\mathbb{R}^{d},\label{eq:model-linear-system}
\end{equation}
 where the $d\times d$ complex matrix $A$ can be diagonalized, $A=P\Lambda P^{-1}$
and $\Lambda=\text{diag}(\lambda_{1},\cdots,\lambda_{d})$. By a change
of variable $U\mapsto P^{-1}U$, the system is decoupled into $d$
linear scalar equations on the complex plane. Consequently, the system
obtained by applying a typical linear numerical integrator can be
diagonalized in the same fashion. We consider using standard one-step
linear integrators as our choice of $F_{H}$ and $C_{H}$. Therefore,
$F_{H}u$ and $C_{H}u$ simply multiply $u$ by suitable complex numbers.
In addition, we will consider $\theta\in\mathbb{C}$, which obviously
commutes with $P$ and $P^{-1}.$ Overall, in this section we consider
initial value problems of the model scalar equation,
\begin{align}
u^{\prime} & =\lambda u,\,\,\,\lambda\in\mathbb{C},\label{eq:Model-linear-scalar-equation}\\
u(0) & =u_{0}.\nonumber 
\end{align}

\paragraph*{Dissipation helps.}

Here, by dissipation, we mean that all eigenvalues of $A$ have negative
real parts. We start by analyzing the standard parareal $(\theta=1)$.
For problems with dissipation, stable and consistent solvers will
naturally have an amplification factors that is strictly less than
one. Suppose that the coarse solver is strictly stable in the sense
that $|C_{H}|\le r_{0}<1,$ 
\[
1<\sum_{j=0}^{m-1}|C_{H}|^{j}=\frac{1-r_{0}^{m}}{1-r_{0}}<\frac{1}{1-r_{0}}.
\]
 Assume that $F_{H}$ and $C_{H}$ are consistent with the same equation,
that $F_{H}$ is $p$-th order method with step size $h\ll H$, and
$C_{H}$ is a $q$-th order method with step size $H$. For $u_{0}$
in a compact subset of the complex plane, 

\[
|F_{H}u_{0}-C_{H}u_{0}|\le K_{1}e^{Re[\lambda H]}h^{p}+K_{2}H^{q+1},
\]
 where $K_{1}$ and $K_{2}$ are two constants that depend on $\lambda$,
the solvers, $h$ and $H$\textbf{. }As a result, the stability of
a standard parareal ($\theta=1$) for $k<n\le N$ requires that,
\[
K_{1}e^{|\lambda|H}h^{p}+K_{2}H^{q+1}<1-r_{0}.
\]
If Re$\lambda<0$, then the terms on the Left Hand Side (LHS) are
bounded in time and the inequality holds for sufficiently small step
sizes. However, with oscillatory problems ($\lambda$ purely imaginary),
the exponential terms may prohibit a large ratio of $H/h$, which
limits the attractiveness of the parareal approach.

Next, we define the amplification factor, which is an upper bound
on the grows of the parareal error.\textbf{
\[
Q_{N^{\prime},k}=|F_{H}-\theta C_{H}|\sum_{j=0}^{N^{\prime}-k-2}|\theta C_{H}|^{j}.
\]
}Comparing with the the error bounds (\ref{eq:Thm23}), it is clear
that the parareal iteration will not be stable unless $Q_{N^{\prime},k}<1$.
We make several observations regarding this bound.
\begin{itemize}
\item The parareal iteration can produce solutions that converge globally
to the one computed by the fine solver, \emph{even if $F_{H}$ and
$C_{H}$ do not solve the same equation. }The iterations will converge
as long as (i) the coarse solver is strictly stable, i.e. $|C_{H}|\le r_{0}<1$,
and (ii) the gap between the fine solver and the coarse solver is
sufficiently small; i.e. $|F_{H}u_{0}-C_{H}u_{0}|<1-r_{0}.$ One simple
way to guaranty that is to choose a coarse solver that can at least
approximately propagate the causality of the given problem. 
\item The above estimates and observations apply when we formally replace
$C_{H}$ by $\theta C_{H}.$ 
\item If the problem is dissipative, then $|C_{H}|<1$, and parareal is
stable as long as $C_{H}$ (or $\theta C_{H}$) are sufficiently close.
However, in general, there exists a maximal value of coarse steps,
$N^{'}$, that depends on inverse powers of $H$ ($h<H$) such that
for $n<N'$ the errors $|e_{n}^{(k)}|$ decreases as $k$ increases,
while for $n\ge N'$, the error $|e_{n}^{(k)}|$ grows exponentially
as $k$ increases. For example, if \textbf{$\theta=1,$ $|C_{H}|=1,$
$k=1$, }then\textbf{ }we\textbf{ }need to pick an $N^{\prime}$ such
that i.e.,\textbf{
\[
Q_{N^{\prime},1}\le K_{3}H^{q+1}N^{\prime}<1.
\]
}Note that this estimate is note sharp. See, for example, Gander and
Hairer \cite{GanderHairer14} provide sharp bounds for the number
of allowed steps in solving Hamiltonian systems using symplectic integrators.
\end{itemize}

\paragraph*{Purely oscillatory problems are more challenging.}

We focus our discussion around the typical case when the coarse solver
is border-line stable; i.e. $|C_{H}|=1$ and $0<\delta\le|F_{H}|\le1$,
where $\delta$ is a lower bound of $|F_{H}|$. A significant implication
is that the parareal iterations will become unstable after several
coarse steps because $\sum_{j=0}^{N-k-2}|C_{H}|^{j}=N-k-1$. In this
case, we see that stability can be gained by multiplying $C_{H}$
by a factor $\theta$. In the extreme case of $\theta=0,$ the errors
in the parareal iterations trivially satisfy,\textbf{
\[
|e_{n}^{(k+1)}|\le|F_{H}|\,|e_{n}^{(k)}|.
\]
}The errors do not necessarily decrease (for $n>k$), unless \textbf{$|F_{H}|<1;$
}i.e. unless the fine solver is strictly linearly stable. For the
case $0<\theta<1$,
\[
1<\sum_{j=0}^{N-k-2}|\theta C_{H}|^{j}=\frac{1-\theta^{N-k-2}}{1-\theta}<\frac{1}{1-\theta}.
\]
The amplification factor is thus bounded by
\begin{align*}
Q_{n,k} & \le|F_{H}-\theta C_{H}|\sum_{j=0}^{n-k-2}|\theta C_{H}|^{j}\\
 & \le\min\left\{ |F_{H}|+\theta,\left(K_{1}h^{p}+K_{2}H^{q+1}+(1-\theta)\right)\right\} \frac{1-\theta^{N-k-2}}{1-\theta},
\end{align*}
The term $K_{1}h^{p}H+K_{2}H^{q+1}$ come from the local errors of
$F_{H}$ and $C_{H}$. When $K_{1}$ and $K_{2}$ are reasonably small,
i.e., the numerical schemes resolves the solution of the differential
equation with sufficiently high accuracy, $Q_{n,k}$ is minimized
for $\theta=1$. However, if the coarse solver does not resolve the
differential equation well, then $K_{2}$ can be very large. This
is the case if $C_{H}$ is some multiscale solver which solves a different
differential equation, or when there are Dirac$\delta$-like impulses
in the system. In such a case, it is reasonable to assume that 
\begin{equation}
|F_{H}|+\theta<K_{1}h^{p}+K_{2}H^{q+1}+(1-\theta).\label{eq:F+theta-smaller}
\end{equation}
In this case, using an appropriate value of $\theta$ may stabilize
the parareal iterations. For example, one may take a $\theta$ that
falls into the range, 
\[
(|F_{H}|+\theta)\frac{1}{1-\theta}<1,\,\,\,i.e.,\,\,\,0\le\theta<\frac{1-|F_{H}|}{2}.
\]
 Again, we see that it is necessary to have $|F_{H}|<1;$ i.e., if
$\theta$ is taken to be a real number then some dissipation in the
fine solver is necessary for stability . 

In the more challenging cases in which $|C_{H}|=1$ and $|F_{H}|=1$,
the parareal iterations needs to be stabilized in another way. We
first look at the following motivating example.
\begin{example}
We consider using the A-stable Trapezoidal rule as both the coarse
and fine integrators, $C_{H}=(1+\lambda H/2)/(1-\lambda H/2)$ and
$F_{H}=\left((1+\lambda h/2)/(1-\lambda H/2)\right)^{H/h}.$ Let $\lambda=-3$
and $H=1$, so that $|C_{H}|=1/5$ and 
\[
\sum_{j=0}^{m-1}|C_{H}|^{j}=\frac{5}{4}(1-\frac{1}{5^{N}})<\frac{5}{4}
\]
 is bounded independent of $N.$ It is clear that parareal iterations
easily converge in this case. Next, consider the oscillatory case
with $\lambda=3i$ and $H=1.$ Now, we have $|C_{H}|=|F_{H}|=1$ and
\[
\sum_{j=0}^{m-1}|C_{H}|^{j}=m.
\]
 We see that the standard parareal algorithm performs poorly compared
to the dissipative case. However, since $C_{H}=(-5+12i)/13$, \emph{it
is possible to multiply }$C_{H}$\emph{ by a complex constant $\theta=re^{i\phi_{H}}$
to minimize the difference $F_{H}u-\theta C_{H}$.} For example, taking
$h=H/20$, $F_{H}=(1591+240i)/1609$ and $\theta=F_{H}/C_{H}\approx-0.24-0.97i$
will drastically improve the convergence and stability of the scheme.
The interpretation is that multiplication by $\theta$ rotates the
coarse solution $C_{H}u$ to have a similar phase as $F_{H}u$. This
is a direct analogy to the ``phase alignment'' procedure proposed
in \cite{AKT-SISC-parareal}. 

In fact, following the old idea of of linear stability of a scheme
for ordinary differential equations, one can systematically look at
the stability property of a ``$\theta$-parareal'' scheme, for $\theta\in\mathbb{C}$, 
\end{example}
\begin{defn}
(Region of parareal-stability) For each $N>0$ and $\xi_{0}\in\mathbb{C}$,
define the set 
\begin{equation}
\mathcal{R}_{\{\lambda H=\xi_{0}\}}^{N}:=\{\theta\in\mathbb{C}:|F_{H}-\theta C_{H}|\sum_{j=0}^{N-2}|\theta C_{H}|^{j}\le1\}.\label{def:stability-region}
\end{equation}
We shall refer to $\mathcal{R}_{\{\lambda H=\xi_{0}\}}^{N}$ as the
region of parareal-stability for the model equation (\ref{eq:Model-linear-scalar-equation}).
Taking $\theta\in\mathcal{R}_{\{\lambda H=\xi_{0}\}}^{N}$ in (\ref{eq:proposed-parareal})
for solving (\ref{eq:Model-linear-scalar-equation}), guaranties that
the resulting errors $|e_{j}^{(k)}|$ will decrease to $0$ as $k$
increases for all $0\le j\le N$ . 
\end{defn}
Figures~\ref{fig:Region-of-parareal-stability} and \ref{fig:Region-of-parareal-stability-dissipative}
show a few examples of the regions of parareal-stability of different
choices $C_{H}$, $F_{H}$, and $H$. We see that for large $|\lambda H|$,
stabilization of the parareal scheme may require $\theta$ to have
non-zero imaginary part; i.e., the coarse solutions need to be \emph{rotated. }

We have seen that, particularly for problems involving oscillations,
the deciding factor for stability and performance of parareal iterations
lies in how well $C_{H}$ approximates $F_{H}$. For oscillatory problems
and ``marginally stable'' integrators, for example in system that
preserve certain energy or invariance, it is necessary to bridge the
gap between the coarse and fine integrators by suitable rotations.
Figure~\ref{fig:Convergence-plots} shows the amplification factors
$Q_{n,0}$ of the standard parareal method involving Forward and backward
Euler schemes and the corresponding cases for the $\theta$-parareal
scheme, with $\theta$ in $\mathcal{R}_{\{\lambda H=0.1i\}}^{N}$.
The results demonstrate that when the number of step is large ($N=10^{4}$
in simulations), stability becomes a critical issue.

Summarizing this example, it is our objective to choose an optimized
choice of $\theta$ to achieve $|F_{H}-\theta C_{H}|\ll|F_{H}-C_{H}|$
while keeping $\sum_{j=0}^{N}|\theta C_{H}|^{j}$ to a moderate size
for some $N$. $\theta C_{H}$ can be viewed as an improved coarse
solver. In the next section, we present two strategies for achieving
this objective. 

\begin{figure}
\begin{centering}
\includegraphics{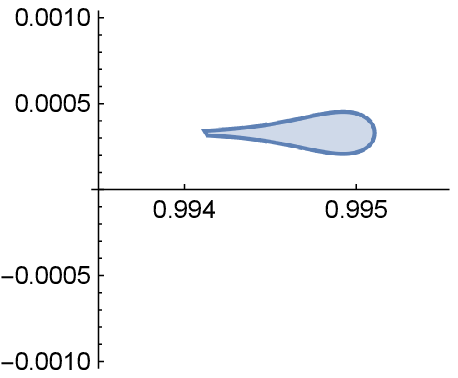}\includegraphics{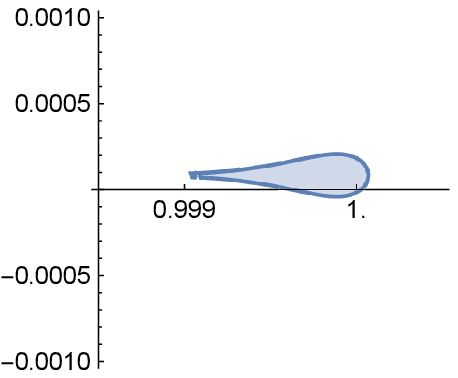}\includegraphics{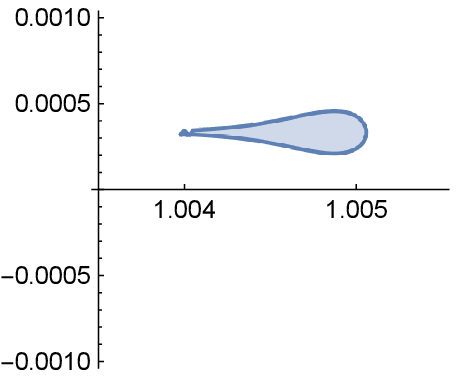}
\par\end{centering}
\caption{Oscillatory example: The region of parareal-stability for (left) forward
Euler, (center) trapezoidal rule and (right) backward Euler as both
the fine and coarse integrators. Parameters are $\lambda H=0.1i$
and $N=10^{4}$. Note that $\theta=1$ is not always included in the
stability region. \label{fig:Region-of-parareal-stability}}
\end{figure}
\begin{figure}
\begin{centering}
\includegraphics{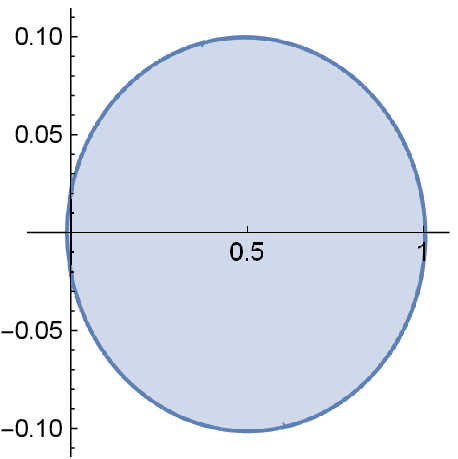}\includegraphics{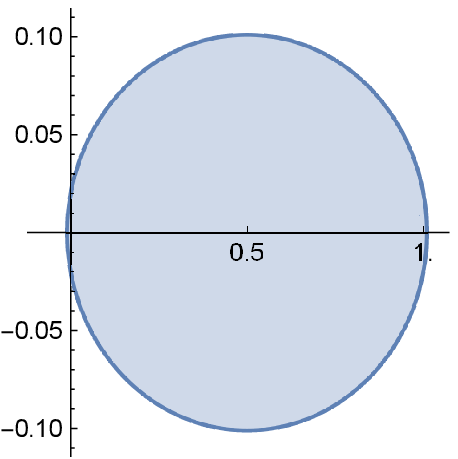}\includegraphics{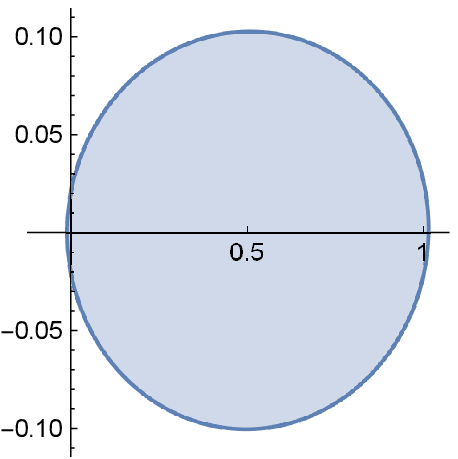}
\par\end{centering}
\caption{Dissipative example: The region of parareal-stability for (left) forward
Euler, (center) trapezoidal rule and (right) backward Euler as both
the fine and coarse integrators. Parameters are $\lambda H=-0.02+0.1i$
and $N=10^{4}$. Note that $\theta=1$ is included in the stability
region, as expected for dissipative systems. \label{fig:Region-of-parareal-stability-dissipative}}
\end{figure}
\begin{figure}
\begin{centering}
\includegraphics{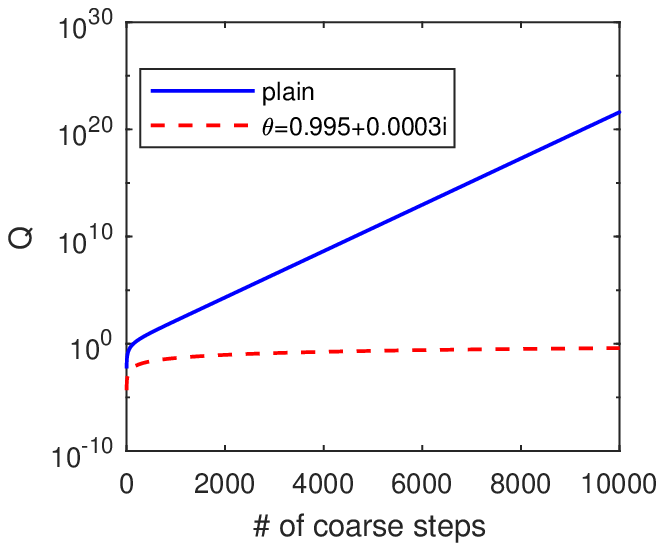}\includegraphics{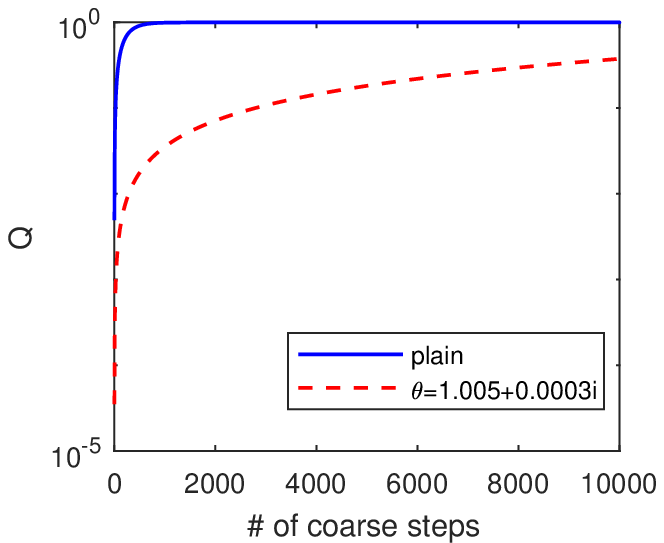}
\par\end{centering}
\begin{centering}
\includegraphics{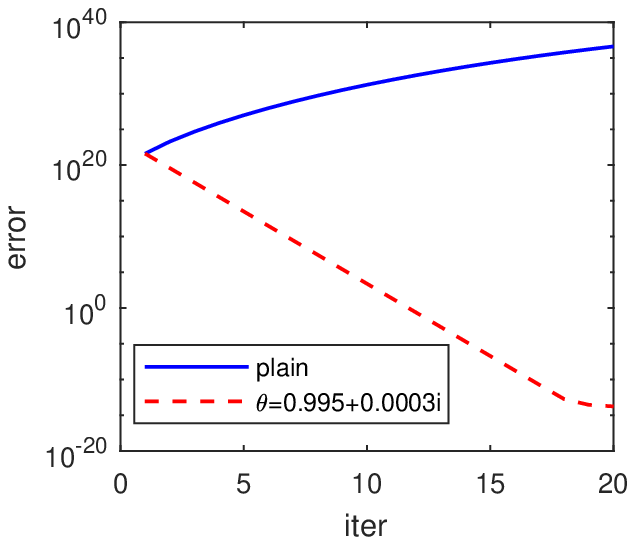}\includegraphics{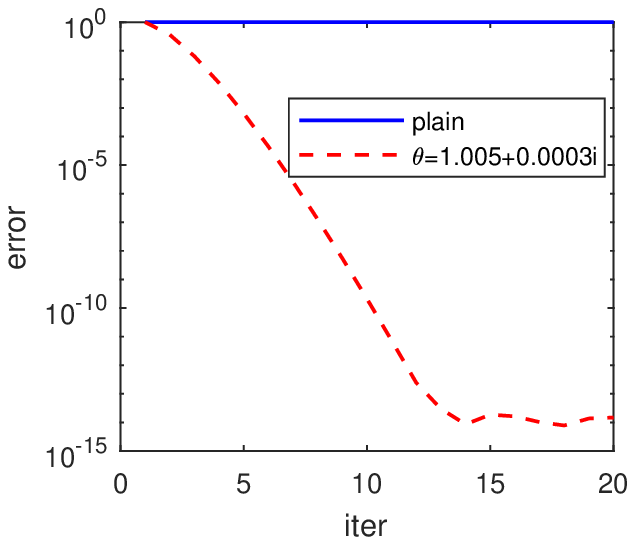}
\par\end{centering}
\caption{\label{fig:Convergence-plots}Convergence of the parareal iterations
for a scalar linear ODE. Top row: The amplification factor $Q_{m,0}=|F_{H}-\theta C_{H}|\sum_{j=0}^{m-2}|\theta C_{H}|^{j}$.
Bottom row shows the computed errors with parameters $\lambda H=0.1i$
and $N=10^{4}$. $F_{H}$ is taken to be the exact solution operator.
The left column shows the results from $C_{H}$ being forward Euler
scheme, and the right column the implicit Euler.}
\end{figure}

\subsection{``Sequentializing'' parareal}

Let $a=|F_{H}-\theta C_{H}|$ and, for oscillatory problems, $b=\sum|\theta C_{H}|^{j}=N-k-2.$
Assuming a coarse solver with step size $H$ and order $q$, $a\sim H^{q+1}$,
which implies, following the stability analysis, that $N<H^{-(q+1)}$.
Therefore, we should not run parareal for more than $N\sim H^{-(q+1)}$
time steps.

In order to overcome this limitation, we propose the divide the the
time interval of interest $[0,T]$ to $s$ subintervals of equal length
and preform parareal sequentially. A similar idea has been suggested
in \cite{GanderHairer14}. As we describe below, this does not significantly
reduce the computational cost for moderate values of $s$.

\paragraph*{Computational cost}

Assume that the coarse and fine integrators apply numerical schemes
using step sizes $H$ and $h$, respectively. Let $n_{CPU}$ denote
the number of CPUs available and assume that $n_{CPU}<T/H$ . Then,
the computational cost of $K$ standard parareal iterations consists
of coarse solver, fine solver operations and extra cost for data transfer
between processors 
\[
C_{p}=K\left(\dfrac{T}{H}+\dfrac{T}{n_{CPU}h}+c_{com}\right)
\]
where $c_{com}$ is the cost of communication. Sequential parareal
processes a shorter time interval at once, $t=T/s$. Then, assuming
that $n_{CPU}<T/(Hs)$, the computational cost of sequential parareal
is, 
\[
C_{sp}=sK\left(\dfrac{T}{Hs}+\dfrac{T}{n_{CPU}hs}+c_{com}\right)\backsimeq C_{p}+sKc_{com}.
\]
Therefore, if the number of processors is not very large (compared
to the maximal theoretical gain using parareal, $T/H$ ) and the communication
in a parallel cluster is efficient, then the computational cost of
sequential parareal is comparable to standard parareal.

\textbf{}%

\subsection{Optimized choices of $\theta$ }

Formula (\ref{eq:proposed-parareal}) and the error estimate (\ref{eq:Thm23})
suggest that $\theta$ needs to bring the coarse scheme to be closer
to the fine one. Indeed, if we can find, without a significant computational
overhead, an operator $\theta$ such that $\theta C_{H}u\equiv F_{H}u$
for all $u\in\mathbb{R}^{d}$, then the parareal scheme is reduced
to simply 
\begin{equation}
u_{n+1}^{(k+1)}=\theta C_{H}u_{n}^{(k+1)}\equiv F_{H}u_{n}^{(k+1)}.\label{eq:ideal-update}
\end{equation}
This means that we shall enjoy the accuracy of the fine integrator. 

In practice, it is more reasonable to \emph{approximate $F_{H}u_{n}^{(k+1)}$
point-wise, using the data gathered from the evaluations of $F_{H}$
at previously computed points $\{u_{n}^{(k)}\}$.} Accordingly, we
shall use the notation $\theta_{n}^{(k)}C_{H}$ to be the approximation
defined near $u_{n}^{(k+1)}$. Of course, the idea of ``recycling''
the computed data to improve the coarse solver is not new. Some existing
parareal algorithms that apply similar approaches have been suggested
and analyzed in \cite{Farhat2003} and \cite{ruprecht2012explicit}.
The point of view of this paper is different as our focus is on increasing
the stability of iterations, even at the cost of high-order accuracy.

Given a choice of coarse and fine integrators, we assume the $\theta_{n}^{(k)}$
is an operator that acts on all previous coarse points \emph{$\{u_{n}^{(j)}\}_{j=0}^{k}$
}into the states space. For simplicity, we will restrict the discussion
to affine linear maps, denoted $\Theta:\mathbb{R}^{d\times(k+1)}\to\mathbb{R}^{d}$
. We consider two approaches in construction of $\theta_{n}^{(k)}$.
The first is a variational approach which directly attempt to minimize
the mismatch between the fine and coarse integrators over a suitable
set of points $\Omega_{n}^{(k)}$ that includes \emph{$\{u_{n}^{(j)}\}_{j=0}^{k}$
}and possibly additional points in its vicinity,

\begin{equation}
\theta_{n}^{(k)}=\arg_{\theta\in\Theta}\text{\ensuremath{\min}}{}_{u\in\Omega_{n}^{(k)}}||F_{H}u-\theta C_{H}u||,\label{eq:optimized-theta}
\end{equation}
where $||\cdot||$ can be the operator or other suitable norm. In
particular, if $C_{H}$ is invertible, then we may simply take 
\[
\theta_{n}^{(k)}\equiv F_{H}C_{H}^{-1}.
\]
We remark that while such approaches may be doable for smaller systems,
it is not practical for large systems unless some low-rank approximation
of $F_{H}$ can be computed efficiently in $\Omega_{n}^{(k)}$. 

\subsection*{Interpolation}

Assuming that $\theta$ is affine and that $\Omega_{n}^{(k)}$ is
a finite set, then the minimization (\ref{eq:optimized-theta}) can
be obtained using linear interpolation. Here, we consider the case
where, $k\ge d$, the set $\Omega_{n}^{(k)}$ includes the last $d+1$
parareal approximations for the $n$'th coarse step, $\Omega_{n}^{(k)}=\left\{ u_{n}^{(k-d)},\dots,u_{n}^{(k)}\right\} $.
Hence, the linear operator $\theta:\mathbb{R}^{d\times(k+1)}\to\mathbb{R}^{d}$
satisfies
\begin{equation}
\theta_{n}^{(k)}C_{H}(u_{n}^{(k-j)})=F_{H}u_{n}^{(k-j)},\,\,\,j=0\dots d,\label{eq:interpolation-constraints}
\end{equation}
 In the next section, we shall present some numerical simulations
using the following construction 
\begin{equation}
\theta_{n}^{(k)}C_{H}(w):=C_{H}w+I_{n}^{(k)}(w;\{u_{n}^{(k-j)}\}_{j=0}^{d}),\label{eq:interp-theta-formula}
\end{equation}
 where $I_{n}^{(k)}(w;\Omega_{n}^{(k)})$ are affine approximations
of the function $\kappa(u):\mathbb{R}^{d}\to\mathbb{R}^{d},$ 
\[
\kappa(u):=\left[F_{H}-C_{H}\right]u,
\]
 which linearly interpolate the points in a set $\Omega_{n}^{(k)}$;
i.e. 
\begin{equation}
I_{n}^{(k)}(u_{n}^{(k-j)};\{u_{n}^{(k-j)}\}_{j=0}^{d})=\kappa(u_{n}^{(k-j)})=\left[F_{H}-C_{H}\right]u_{n}^{(k-j)},\,\,\,j=0\dots d.\label{eq:interp-criteria}
\end{equation}
For brevity, we shall write below $I_{n}^{(k)}w=I_{n}^{(k)}(w;\{u_{n}^{(k-j)}\}_{j=0}^{d})$.
Using, (\ref{eq:interp-theta-formula}) and (\ref{eq:interp-criteria}),
\[
\theta_{n}^{(k)}C_{H}u_{n}^{(k)}=C_{H}u_{n}^{(k)}+I_{n}^{(k)}u_{n}^{(k)}=C_{H}u_{n}^{(k)}+\left[F_{H}-C_{H}\right]u_{n}^{(k)}=F_{H}u_{n}^{(k)}.
\]
Substituting into the$\theta$-parareal update,

\begin{equation}
u_{n+1}^{(k+1)}=\theta_{n}^{(k)}C_{H}u_{n}^{(k+1)}+F_{H}u_{n}^{(k)}-\theta_{n}^{(k)}C_{H}u_{n}^{(k)}=\theta_{n}^{(k)}C_{H}u_{n}^{(k+1)}=\left[C_{H}+I_{n}^{(k)}\right]u_{n}^{(k+1)}.\label{eq:compact-interp-theta-scheme}
\end{equation}

\subsubsection{Error estimates}

Let us consider the interpolative approach from the view point of
a linear approximation of the function $\kappa(u).$ An ``ideal''
parareal update (\ref{eq:ideal-update}), in the sense that $u_{n}^{(k)}=F_{H}^{n}u_{0}$,
could be written as 
\begin{equation}
u_{n+1}^{(k+1)}=C_{H}u_{n}^{(k+1)}+\kappa(u_{n}^{(k+1)}).\label{eq:ideal_parareal_kappa}
\end{equation}
However, $\kappa(u_{n}^{(k+1)})$ are not known and need to be approximated.
In this formulation, standard parareal amounts to approximating $\kappa(u_{n}^{(k+1)})$
by $\kappa(u_{n}^{(k)})$. We may make use of Taylor series to design
higher order approximation of $\kappa$ near the computed values.
Expanding $\kappa$ around $u_{n}^{(k)},$we have

\begin{equation}
\kappa(w)=\kappa(u_{n}^{(k)})+D_{\kappa}(u_{n}^{(k)})(w-u_{n}^{(k)})+R_{2}(w-u_{n}^{(k)}),\label{eq:kappa_expansion}
\end{equation}
where $R_{2}$ is the second order (in the sense of small distance
to $u_{n}^{(k)})$ remainder term. In this way, the ``ideal'' parareal
update then formally becomes,

\begin{equation}
u_{n+1}^{(k+1)}=C_{H}u_{n}^{(k+1)}+\kappa(u_{n}^{(k)})+D_{\kappa}(u_{n}^{(k)})(u_{n}^{(k+1)}-u_{n}^{(k)})+R_{2}(u_{n}^{(k+1)}-u_{n}^{(k)}).\label{eq:interp_parareal_update}
\end{equation}
We may define a parareal update by truncating the higher order terms,
$R_{2}$ in the expansion of $\kappa$ around $u_{n}^{(k)}$. The
following Lemma shows that $\kappa(u_{n}^{(k)})+D_{\kappa}(u_{n}^{(k)})(w-u_{n}^{(k)})$
can be approximated to second order by linear interpolation. 
\begin{lem}
Let $x\in\mathbb{R}^{d}$ and let $\Omega=\left\{ x_{0},\dots,x_{d}\right\} \subset\mathbb{R}^{d}$
be a set of $d+1$ interpolation points in $B_{\epsilon}(x)=\left\{ w:\left|w-x\right|<\epsilon\right\} $
for some $\epsilon>0$. Let 
\[
V=\left[\begin{array}{cccc}
x_{0} & \;x_{1} & \;\cdots & \;x_{d}\\
1 & 1 & \;\cdots & \;1
\end{array}\right],
\]
a $(d+1)\times(d+1)$ matrix, and assume that $\det V\neq0$. Let
$\kappa:\mathbb{R}^{d}\to\mathbb{R}^{d}$ denote a $C^{2}$ function
in the $\epsilon$-ball. Let $l:\mathbb{R}^{d}\to\mathbb{R}^{d}$
denote the unique affine function, $l(w)=l_{0}+Aw$, such that $\kappa(w)=l(w)$
for all $w=x_{0}\dots x_{d}$. Then,
\[
\left|\kappa(x)-l(x)\right|\le C\frac{\epsilon^{2}}{\sigma_{\min}},
\]
for some constant $C>0$, where $\sigma_{\min}$ is the smallest singular
value of $V$.
\end{lem}
\begin{proof}
Without loss of generality, we take $x=0$. Denote
\[
l(w)=L\left[\begin{array}{c}
w\\
1
\end{array}\right],
\]
 where $L$ is a $d\times(d+1)$ matrix that satisfies 
\[
LV=F=\left[\kappa(x_{0})\;\cdots\;\kappa(x_{d})\right].
\]
Since $V$ is invertible, $L:=FV^{-1}$ is uniquely defined. Also,
expanding $\kappa$ in a Taylor series around $0$ and evaluating
at $x_{j},$
\[
\kappa(x_{j})=\kappa(0)+\nabla\kappa(0)x_{j}+O(\epsilon^{2}),
\]
or, arranging these $d+1$ columns in a matrix,
\[
F=\kappa(0)\left[1\;\cdots\;1\right]+\nabla\kappa(0)\left[x_{0}\;\cdots\;x_{d}\right]+O(\epsilon^{2})=\left[\nabla\kappa(0)\;\kappa(0)\right]V+O(\epsilon^{2}).
\]
Multiplying by $V^{-1}$ we conclude that $L=\left[\nabla\kappa(0)\;\kappa(0)\right]+O(\epsilon^{2}/\sigma_{\min})$.
As a result,
\[
l(0)=\left[\nabla\kappa(0)\;\kappa(0)\right]\left[\begin{array}{c}
0\\
\vdots\\
0\\
1
\end{array}\right]+O(\epsilon^{2}/\sigma_{\min})=\kappa(0)+O(\epsilon^{2}/\sigma_{\min}),
\]
which proves the Lemma.
\end{proof}
Returning to interpolative $\theta$-parareal, we take $x_{j}=u_{n}^{(k-j)}$.
Assume that, at time step $n$, parareal reached accuracy $\epsilon$
in the last $d+1$ iterations, i.e.,
\[
\max_{j=0\dots d}|u_{n}^{(k-j)}-F_{H}^{n}u_{0}|<\epsilon.
\]
Then, using Lemma 2.6 with $x=u_{n}^{(k+1)}$, the linear interpolation
$I_{n}^{(k)}(w;\left\{ u_{n}^{(k+1)}\right\} _{j=0}^{d})$ satisfies,

\[
I_{n}^{(k)}u_{n}^{(k+1)}=\kappa(u_{n}^{(k+1)})+O(\epsilon^{2}/\sigma_{\min}).
\]
Using (\ref{eq:ideal_parareal_kappa}) implies that
\begin{equation}
|u_{n+1}^{(k+1)}-F_{H}^{n+1}u_{0}|=O(\epsilon^{2}/\sigma_{\min}).\label{eq:interpolative-parareal-error-estimate}
\end{equation}
We conclude that, after every $d+1$ iterations of the interpolative
$\theta$-parareal, the error $\epsilon$ is reduced to $O(\epsilon^{2}/\sigma_{\min})$.
We keep $\sigma_{\text{min }}$, the smallest singular value of $V$,
in the above formula to reflect a potential problem when $V$ is not
well-conditioned. Viewing from another angle, it also suggest an opportunity
in considering interpolation in lower dimensional subspaces. 

\subsubsection{Interpolation in lower dimensional subspaces}

Starting from a given initial condition, structure preserving schemes
will produce numerical solutions which lie on certain invariant manifolds,
immersed in the higher dimensional phase space. There are two possible
simple approaches that could be used to make the interpolation well-defined:
\begin{enumerate}
\item Add a suitable number of additional data points near $u_{n}^{(k)}$
to allow a unique linear interpolation.
\item Interpolate $\kappa(u)=F_{H}u-C_{H}u$ in a lower dimensional subspace.
\end{enumerate}
In this paper, we discuss the second strategy. For convenience, we
shall present the algorithm is a slightly different setup. Denote
\[
W_{n}^{(k)}:=\left[\begin{array}{ccc}
u_{n}^{(k)} & \;\cdots & \;u_{n}^{(k-d)}\\
1 & \cdots & 1
\end{array}\right].
\]
The main idea is to work with the singular value decomposition of
$W_{n}^{(k)}.$

Assume that the singular values of the matrix $W_{n}^{(k)}$ satisfy
$\sigma_{j}\le\epsilon,\,\,\,\ell<j\le d,$ where $\epsilon$ is a
chosen tolerance. Let $\tilde{U}\tilde{\Sigma}\tilde{V}^{T}$ be the
truncated singular value decomposition corresponding to $\sigma_{1}\dots\sigma_{\ell}.$
Then, interpolation can be performed in the subspace spanned by the
first $\ell$ singular vectors,

\[
\tilde{I}_{n,\Delta}^{(k)}:=\left[\begin{array}{ccc}
\kappa(u_{n}^{(k)}) & \;\cdots & \;\kappa(u_{n}^{(k-j)})\end{array}\right]\tilde{V}\tilde{\Sigma}^{-1}\tilde{U}^{T}.
\]
The interpolative parareal iteration (\ref{eq:compact-interp-theta-scheme})
becomes
\[
u_{n+1}^{(k+1)}:=C_{H}u_{n}^{(k+1)}+\tilde{I}_{n,\Delta}^{(k)}\left[\begin{array}{c}
u_{n}^{(k+1)}\\
1
\end{array}\right].
\]
 To further improve stability, we may consider abandoning the approximation
of $\kappa(u)$ if the interpolated values are too far away from the
previously computed ones. For example, in the computations presented
in the next section, we use the criterion
\[
\|\tilde{I}_{n,\Delta}^{(k)}(u_{n}^{(k+1)})\|<2\max_{j}\|\kappa(u_{n}^{(j)})\|.
\]
If the above criterion is not met by the computation, the algorithm
switches back to the plain parareal which uses a lower order approximation
of $\kappa.$ We summarize our proposed method in Algorithm~\ref{alg:Interp_based_parareal}.

\begin{algorithm}
\begin{enumerate}
\item $k=0$: Compute the coarse approximation $u_{n}^{(0)}=C_{H}u_{n-1}^{(0)}$.
Set $\left[\begin{array}{c}
W_{n}^{(0)}\end{array}\right]=\left[\begin{array}{cccc}
u_{n}^{(0)} & \;u_{n}^{(0)} & \;\cdots & \;u_{n}^{(0)}\end{array}\right].$
\item $k=1$: Compute $\kappa(u_{n}^{(0)})=F_{H}u_{n-1}^{(0)}-C_{H}u_{n-1}^{(0)}$
in parallel. 

$K_{n}^{(0)}=\left[\begin{array}{cccc}
\kappa(u_{n}^{(0)}) & \;\kappa(u_{n}^{(0)}) & \;\cdots & \;\kappa(u_{n}^{(0)})\end{array}\right].$

$u_{n}^{(1)}=C_{H}u_{n-1}^{(1)}+\kappa(u_{n}^{(0)}).$

$\left[\begin{array}{c}
W_{n}^{(1)}\end{array}\right]=\left[\begin{array}{cccc}
u_{n}^{(1)} & \;u_{n}^{(0)} & \;\cdots & \;u_{n}^{(0)}\end{array}\right].$
\item For $k\geq2$, compute $\kappa(u_{n}^{(k-1)})$ in parallel. 

$K_{n}^{(k-1)}=\left[\begin{array}{ccc}
\kappa(u_{n}^{(k-1)}) & | & K_{n}^{(k-2)}\left[:,2:d\right]\end{array}\right].$
\begin{enumerate}
\item Compute SVD of $\left[\begin{array}{c}
W_{n}^{(k-1)}\end{array}\right]=U\Sigma V^{T}$. 
\item Select $m$ largest singular values such that $\nicefrac{\sigma_{1}}{\sigma_{m}}>tol$.\\
(In numerical examples, we set $tol=10^{-14}$.) \\
Define $\tilde{U}=U\left[:,1:m\right]$, $\tilde{V}=V\left[:,1:m\right]$,
$\tilde{\Sigma}=\Sigma\left[1:m,1:m\right]$.
\item $\theta$-Parareal update:

If $m=1$: 

\ \ \ \ $u_{n}^{(k)}=C_{H}u_{n-1}^{(k)}+\kappa(u_{n}^{(k-1)})$. 

Else:

\ \ \ \ $I_{n,\Delta}^{(k-1)}=K_{n}^{(k-1)}\tilde{V}\tilde{\Sigma}^{-1}\tilde{U}^{T}.$ 

\ \ \ \ $\kappa^{*}=I_{n,\Delta}^{(k-1)}\left[\begin{array}{c}
u_{n-1}^{(k)}\\
1
\end{array}\right].$

\ \ \ \ If $\kappa^{*}>2\max_{j}\left\{ \|K_{n}^{(k-1)}\left[:,j\right]\|\right\} $:

\ \ \ \ \ \ \ \ $u_{n}^{(k)}=C_{H}u_{n-1}^{(k)}+\kappa(u_{n}^{(k-1)})$. 

\ \ \ \ Else:

\ \ \ \ \ \ \ \ $u_{n}^{(k)}=C_{H}u_{n-1}^{(k)}+v_{n}^{(k-1)}$.

\ \ \ \ End

End
\item Update the matrix: $W_{n}^{(k)}=\left[\begin{array}{ccc}
\begin{array}{c}
u_{n}^{(k)}\\
1
\end{array} & | & W_{n}^{(k-1)}\left[:,2:d\right]\end{array}\right].$
\end{enumerate}
\end{enumerate}
\caption{\label{alg:Interp_based_parareal}Lower dimensional interpolation
based $\theta$-parareal algorithm. }
\end{algorithm}

\section{Numerical examples}

We demonstrate some of the properties of the proposed scheme in two
areas: (1) multiscale coupling and (2) simulation of Hamiltonian systems
in long time intervals. We focus on the case in which the number of
coarse steps is significantly larger than allowed by the stability
analysis of the standard parareal, i.e., $N\gg H^{-q-1}$.
\begin{example}
A single harmonic oscillator. Consider,
\[
\left[\begin{array}{c}
q'\\
p'
\end{array}\right]=\left[\begin{array}{cc}
0 & 1\\
-\omega^{2} & 0
\end{array}\right]\left[\begin{array}{c}
q\\
p
\end{array}\right]
\]
where $\omega$ is the stiffness constant. Figure \ref{fig:harmonic_error}
depicts result obtained using Velocity Verlet for both the fine and
coarse integrators with step sizes $h$ and $H$, respectively. Initial
conditions are $q_{0}=1,\;p_{0}=-1$. Denoting, 
\[
Q_{\Delta t}=\left[\begin{array}{cc}
1-\dfrac{1}{2}\omega^{2}\Delta t^{2} & \Delta t\\
-\omega^{2}\Delta t+\dfrac{1}{4}\omega^{4}\Delta t^{3} & \;1-\dfrac{1}{2}\omega^{2}\Delta t^{2}
\end{array}\right],
\]
 yields, $F_{H}=$ $(Q_{h})^{H/h}=P\Lambda_{h}^{H/h}P^{-1}$ and $C_{H}=Q_{H}.$
Here, $P$ is the diagonalizing matrix. Hence, in this example the
optimal value of $\theta$ is constant, $\theta=F_{H}C_{H}^{-1}$. 
\end{example}
\begin{figure}
\begin{centering}
\includegraphics{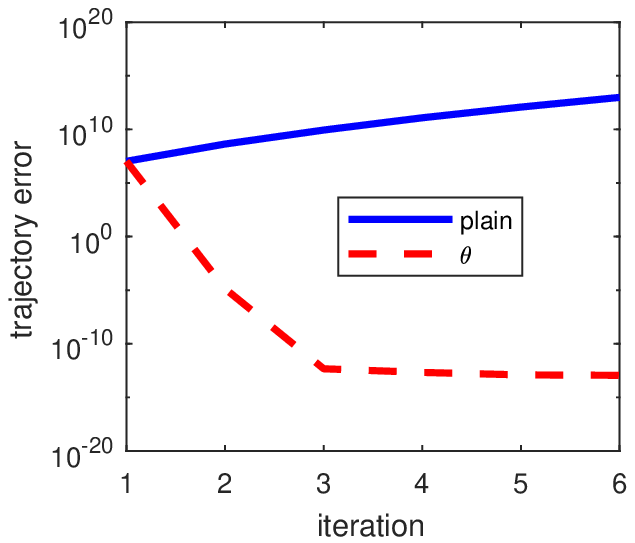}\includegraphics{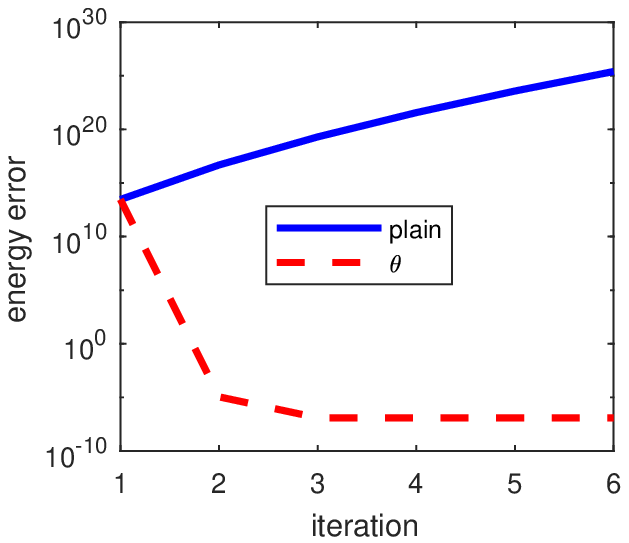}
\par\end{centering}
\caption{\label{fig:harmonic_error}Maximum error of standard (solid) and $\theta$-parareal
(dashed) on harmonic oscillator. The stiffness is $\omega=1$ and
final time is $T=10^{3}$. Fine and coarse integrator are Velocity
Verlet with step sizes $H=0.5,\;h=10^{-2}$.}
\end{figure}

\phantom{}
\begin{example}
Inhomogeneous linear system with variable-coefficients and singular
pulses-like forcing. A low dimensional example. Consider the following
forced linear system with varying coefficients,
\end{example}
\[
\left[\begin{array}{c}
x'\\
y'
\end{array}\right]=A(t)\left[\begin{array}{c}
x\\
y
\end{array}\right]+b(t),
\]
where $A(t)$ is a time dependent matrix of purely imaginary eigenvalues,
and $b(t)$ is an external force. Figure \ref{fig:variedcoefficient_error},
depicts results using the midpoint rule as a coarse integrator and
fourth order Runge-Kutta (RK4) as a fine one. Initial conditions are
$x_{0}=1,\;y_{0}=0$. The coefficient matrix is $A(t)=\left[\begin{array}{cc}
0 & 1\\
-(\cos(t)^{2}+1) & 0
\end{array}\right]$ and the forcing term is $b(t)=\sum_{i=1}^{40}e^{-50(t-t_{i})^{2}}$
where $t_{i}$ are chosen randomly in $[0,T]$. Similar to the example
above, the optimal $\theta$ is given by $\theta=\tilde{F}_{H}\tilde{C}_{H}^{-1}$.
However, in this example it is time-dependent. Disregarding the forcing
term, 
\[
\tilde{C}_{H}=\boldsymbol{1}+H\left(A(t)+\dfrac{1}{2}HA(t+\dfrac{1}{2}H)\right)
\]
\[
\tilde{F}_{H}=\left[\boldsymbol{1}+hA(t)+\dfrac{h^{2}}{2}A\left(t+\dfrac{h}{2}\right)A(t)+\dfrac{h^{3}}{6}A^{2}\left(t+\dfrac{h}{2}\right)A(t)+\dfrac{h^{4}}{24}A(t+h)A^{2}\left(t+\dfrac{h}{2}\right)A(t)\right]^{H/h}.
\]
Thus, $\theta$ needs to be computed at every coarse time step. However,
it is the same for all iterations. 
\begin{figure}
\begin{centering}
\includegraphics{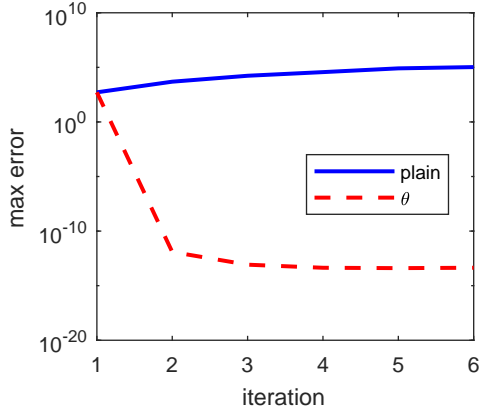}
\par\end{centering}
\caption{\label{fig:variedcoefficient_error}Error of standard (solid) and
$\theta$-parareal (dashed) methods for the variable coefficient system
with time varying frequencies and forcing, example 3.2. The coarse
integrator is midpoint with stepsize $H=0.5$ and fine integrator
is RK4 with stepsize $h=10^{-2}$. Final time $T=100$.}
\end{figure}

\phantom{}
\begin{example}
``De-homogenization''. In this example, we show how parareal can
be used to ``fill-in'' the details in multiscale numerically homogenized
solutions. Again, the motivation of this high-dimensional example
is to demonstrate how the parareal scheme can be stabilized using
a simple multiplication by a small real number $\theta$. The main
point out of this example is not about an optimal solution of the
PDE, but rather the feasibility of using the parareal framework for
multiscale computation.

We consider the following heat equation with highly oscillatory coefficient,
\begin{equation}
u_{t}^{\epsilon}=\frac{\partial}{\partial x}\left(a(x,\frac{x}{\epsilon})u_{x}^{\epsilon}\right),\,\,\,0\le x\le1,t>0,\label{eq:heat-osc-coeff}
\end{equation}
 with 
\[
a(x,y)=a_{0}+\sin(2\pi y)(1-e^{-100(x-0.133104)^{2}}),\,\,\,a_{0}=1.1\,\text{or}\,1.01.
\]
Initial condition are $u^{\epsilon}(x,0)=x(1-x)$ and $u(0)=u(1)=0.$
\end{example}
The fine integrator $F_{H}$ evolves the discretized system derived
from centered differencing for the right hand side of the differential
equation using the classical 3-point stencil on a uniform mesh, $\{x_{j}=j\Delta x:j=1,2,\cdots,M-1\}$,
with $\Delta x=\epsilon/20,\epsilon=0.04$. In parallel, each fine
integration runs 50 Crank-Nicholson (CN) steps with step size $h=H/50$,
with $H=2\Delta x$. The coarse integrator solves the homogenized
equation 
\begin{equation}
\bar{u}_{t}-\bar{A}\ \bar{u}_{xx}=0,\ \ \bar{A}=\begin{cases}
\sqrt{0.21}, & a_{0}=1.1,\\
0.141774, & a_{0}=1.01,
\end{cases}\label{eq:heat-homogenized}
\end{equation}
on the same spatial grid, with the same initial and boundary conditions
as above. Time steps are either Implicit Euler (IE) or CN, running
$100$ coarse steps using $H=2\Delta x$. 

Denoting the discrete solution as $u_{n}^{(k)}=(u_{1,n}^{(k)},u_{2,n}^{(k)},\cdots,u_{M-1,n}^{(k)}),$
where $u_{j,n}^{(k)}$ is the computed solution at grid node $x_{j}$
and time $nH$ at the $k$-th parareal iteration, the relative errors
are given by
\[
e_{n}^{(k)}:=\frac{||u_{n}^{(k)}-u_{n}^{\epsilon}||_{h}}{||u_{n}^{\epsilon}||_{h}},
\]
where $u_{n}^{\epsilon}$ is the reference solution at $t=nH$, computed
using Crank-Nicholson with step size $h$. The norm $||\cdot||_{h}$
is the usual 2-norm for grid functions, e.g. $||u_{n}^{(k)}||_{h}:=\left(\sum_{j=1}^{M-1}|u_{j,n}^{(k)}|^{2}\Delta x\right)^{1/2}$.
We report the errors at $n=75,$ which is at 3/4 of the total simulated
time steps. The purpose is simply to avoid being too close to the
largest parareal iterations that we simulate. 

Standard parareal ($\theta=1$) yields unstable iterations with both
IE and CN, unless there is dissipation. Figure \ref{fig:Stability-of-multiscale-coupling}
depicts a comparison, with dissipation term, between different values
of $\theta$ applied to the IE or CN schemes as coarse time integrators.
CN scheme require smaller values of $\theta$ (i.e., the stability
region is narrower). On the Fourier domain, we see that the amplification
factor for IE is 
\[
\hat{Q}_{IE}(\omega)=\frac{1}{1+4\sigma\sin^{2}\xi/2},
\]
 and that of CN is 
\[
\hat{Q}_{CN}(\omega)=\frac{1-2\sigma\sin^{2}\xi/2}{1+2\sigma\sin^{2}\xi/2},
\]
 where $\sigma=\bar{A}\Delta t/\Delta x^{2}$ and $\xi=\omega\Delta x$.
We see first that the parareal iterations with IE are more stable
because $|\hat{Q}_{IE}|$ is smaller than $|\hat{Q}_{CN}|$ in general.
Furthermore, for large $\sigma,$ $\hat{Q}_{CN}$ is close to $-1$
for $\xi\neq0$. In our simulations, we used $\sigma=2\Delta x^{-1}$
(i.e. $\Delta t\sim\Delta x$), thus the coupling is highly unstable
because $|C_{H}|\approx1$. Figure \ref{fig:Stability-of-multiscale-coupling}
shows that a smaller value of $\theta$ is needed to stabilize the
parareal iterations with CN. The price of using a smaller $\theta$
is that $|F_{H}-\theta C_{H}|$ is bigger and the overall amplification
factor is not as small as that using IE. 

\begin{figure}
\begin{centering}
\includegraphics{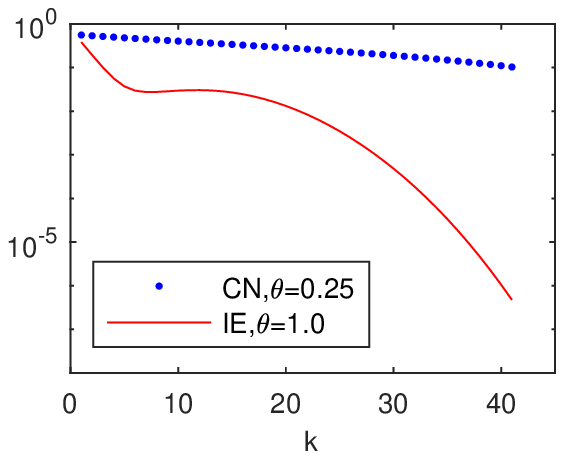}\includegraphics{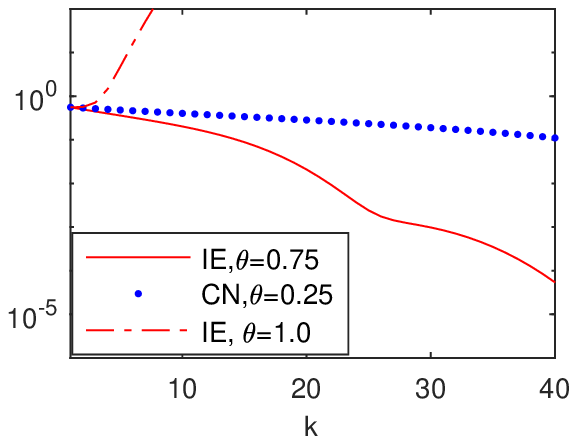}
\par\end{centering}
\caption{\label{fig:Stability-of-multiscale-coupling}Stability of the multiscale
couplings, considered in Example 3.3. The two subplots show the relative
errors computed by different values of $\theta$ applied to the IE
(Implicit Euler) and CN (Crank-Nicolson) schemes as coarse time integrators.
The left subplot shows the relative errors for the case $a_{0}=1.1$
and the right subplot shows the errors for the case $a_{0}=1.01.$
The latter case is less diffusive and corresponding requires more
stabilization.  }
\end{figure}

\begin{example}
Linear wave equation
\[
u_{tt}=c^{2}(x)\Delta u,\,\,\,0\le x<1,t\ge0
\]
 with periodic boundary condition in $x$ and the initial conditions
\[
u(x,0)=u_{0}(x):=0.1\exp(-50(x-0.5)^{2})
\]
 and $u_{t}(x,0)=0.$ The coarse solver will solve this problem using
the wave speeds $c(x)=\bar{c}(x)\equiv1$, and the fine solver will
use the wave speed 
\begin{equation}
c^{2}(x)=1-0.2e^{-2000(x-0.133104)^{2}}-0.1e^{-2000(x-0.733104)^{2}}.\label{eq:wave-speed-and-bumps}
\end{equation}
 Both shall use the following notations to denote the numerical approximation
to the solution $u(x,t)$ and $u_{t}(x,t$): 
\[
\mathbf{u}_{n,j}=\left(\begin{array}{c}
u_{n,j}\\
p_{n,j}
\end{array}\right)\approx\left(\begin{array}{c}
u(j\Delta x,n\Delta t)\\
u_{t}(j\Delta x,n\Delta t)
\end{array}\right),
\]
 where $\Delta x$ and $\Delta t$ are respectively the grid spacings
in $x$ and in $t$ used in the finite difference scheme
\begin{align}
p_{n+1,j} & =p_{n,j}+\Delta t\,c^{2}(j\Delta x)D_{+}^{x}D_{-}^{x}u_{n,j}-\alpha\Delta t\Delta x^{3}D_{-}^{t}(D_{+}^{x}D_{-}^{x})^{2}u_{n,j},\label{eq:wave-scheme-p}\\
u_{n+1,j} & =u_{n,j}+\Delta t\,p_{n+1,j},\label{eq:wave-scheme-u}
\end{align}
 with initial conditions $u_{0,j}=u_{0}(j\Delta x)$ and $p_{0,j}=0,$
$j=0,1,\cdots,\Delta x^{-1}-1$ and the boundary condition $u_{n,M}=u_{n,0},$
$n=0,1,2,\cdots,N.$ We will use $\alpha<1/15.$ The last term in
(\ref{eq:wave-scheme-p}) is a discretization of the damping term
$\alpha\Delta x^{3}u_{txxxx}$ which damps out very high frequency
Fourier components of solutions. The stability condition for this
scheme requires that $\Delta t\le\Delta x/2.$ We shall use the same
scheme for both the coarse and the fine solver. The only difference
is that the coarse solver will solve on a grid with spacing $\Delta x=H,$
$\Delta t=H/2$, and the fine on a grid with $\Delta x=H/40,$ and
$\Delta t=H/80.$ 

We present numerical results computed by the following iterations:

\begin{equation}
\mathbf{u}_{n+1}^{(k+1)}=\theta_{n}^{(k)}C_{H}\mathbf{u}_{n+1}^{(k+1)}+\mathcal{P}F_{H}\mathcal{R}\mathbf{u}_{n}^{(k)}-\theta_{n}^{(k)}C_{H}\mathbf{u}_{n}^{(k)},\,\,\,0<\theta_{n}^{(k)}\le1.\label{eq:theta-parareal-for-wave-eqn}
\end{equation}
 Here $\mathcal{R}$ is the \emph{reconstruction operator} that takes
a grid function defined on $H\mathbb{Z}\cap[0,1)$ to a grid function
defined on the finer grid $h\mathbb{Z}\cap[0,1)$; $\mathcal{P}$
is the projection operator that maps the grid function defined on
$h\mathbb{Z}\cap[0,1)$ to a grid function on the coarser grid $H\mathbb{Z}\cap[0,1)$.
In the following simulations, $\mathcal{R}$ is defined by the cubic
interpolation that assuming the grid function to be periodic on $[0,1),$
while $\mathcal{P}$ is simply taken to be the pointwise restriction
assuming that $H$ is divisible by $h$. 

In Figures~\ref{fig:wave-result-1} and \ref{fig:wave-result-errors},
we present a result computed using
\[
\theta_{n}^{(k)}=\begin{cases}
1, & n\le750\,\,\,\text{and}\,\,\,k>3,\\
1-\frac{3}{4}(nH)\,10^{-3}, & \text{otherwise.}
\end{cases}
\]
 The simulation involves $800$ coarse steps. We found that both the
stabilization term (the last term in (\ref{eq:wave-scheme-p})) as
well as $\theta_{n}^{(k)}$ being smaller than 1 for large $n$ play
important role in the stability of the parareal iterations. Standard
parareal scheme typically become very unstable in the setup considered
in this example. We also observe that even though the initial errors
is improved by over 90\% after only few iterations, improvement by
further iterations is rather small. More elaborate stabilization is
required if one wishes to speed up the convergence rate. 

\begin{figure}

\begin{centering}
\includegraphics{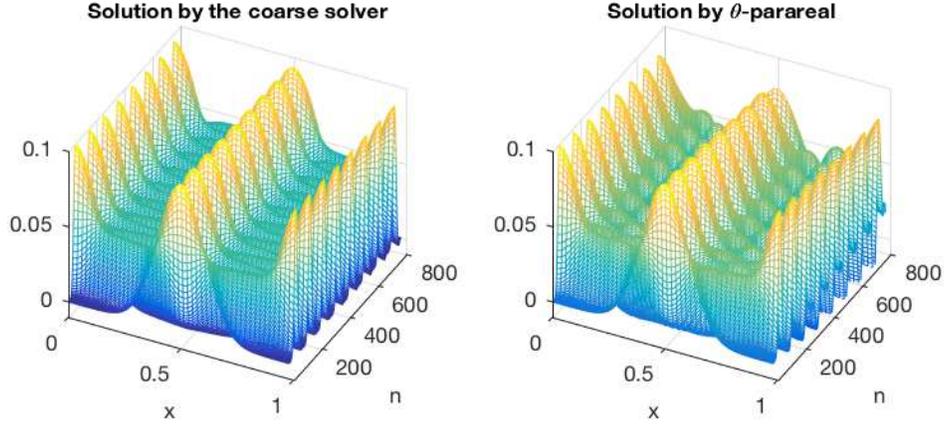}
\par\end{centering}
\caption{\label{fig:wave-result-1}Example 3.4 (wave equation). Left: Solution
computed by the coarse solver, without parareal coupling to the fine
solver. Right: The $\theta$-parareal solution computed at $k=12$.}

\end{figure}
\begin{figure}
\begin{centering}
\includegraphics{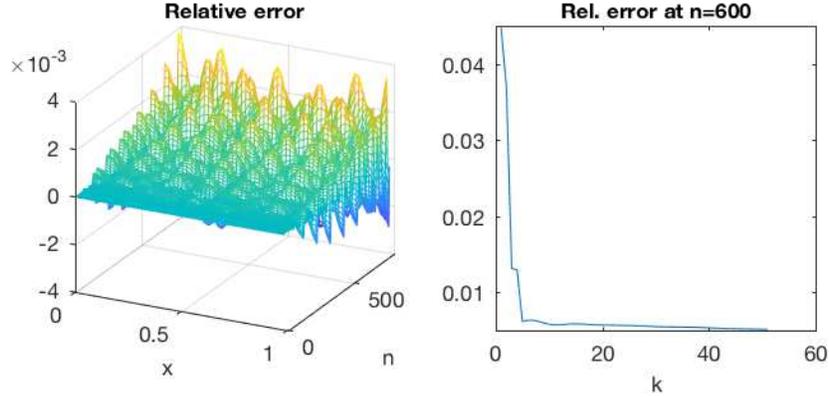}
\par\end{centering}
\caption{\label{fig:wave-result-errors}Example 3.4 (wave equation). Left:
Pointwise errors of the $\theta$-parareal solution at $k=5$. Right:
The relative error at $T_{n}=12$ (i.e. 600 coarse steps) as a function
of the parareal iteration number $k$.}
\end{figure}
\end{example}
\phantom{}
\begin{example}
A spin orbit example. The following example low-dimensional Hamiltonian
system as been studied in \cite{Jimenez2011}.
\[
\left[\begin{array}{c}
q'\\
p'
\end{array}\right]=\left[\begin{array}{c}
p\\
-2\epsilon\sin(q)-2\alpha\sin(2q+\phi)+14\alpha\sin(2q-\phi)
\end{array}\right].
\]
The initial condition is $\left[q_{0},p_{0}\right]=\left[1,0\right]$.
Figure \ref{fig:spinorbit_error} presents results for standard and
$\theta$-parareal where the optimal value $\theta$ is approximated
at each iteration and time step using the interpolation method described
in Algorithm \ref{alg:Interp_based_parareal}. Standard parareal is
unstable after a large number of steps.
\end{example}
\begin{figure}
\begin{centering}
\includegraphics{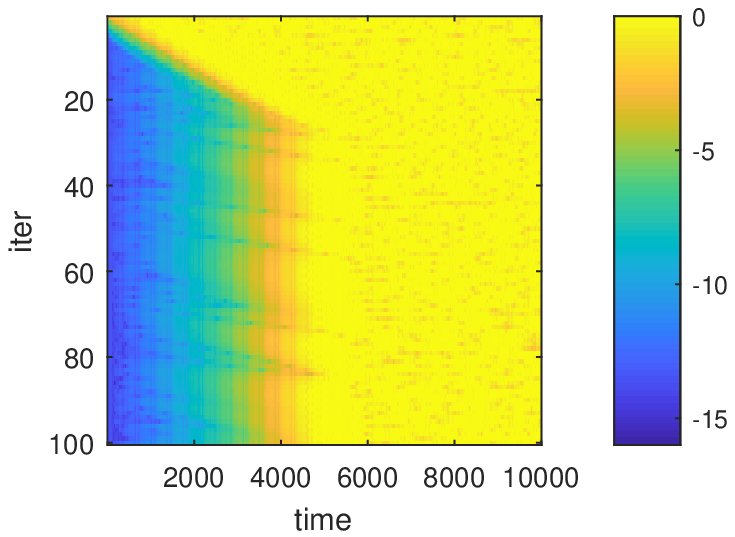}\includegraphics{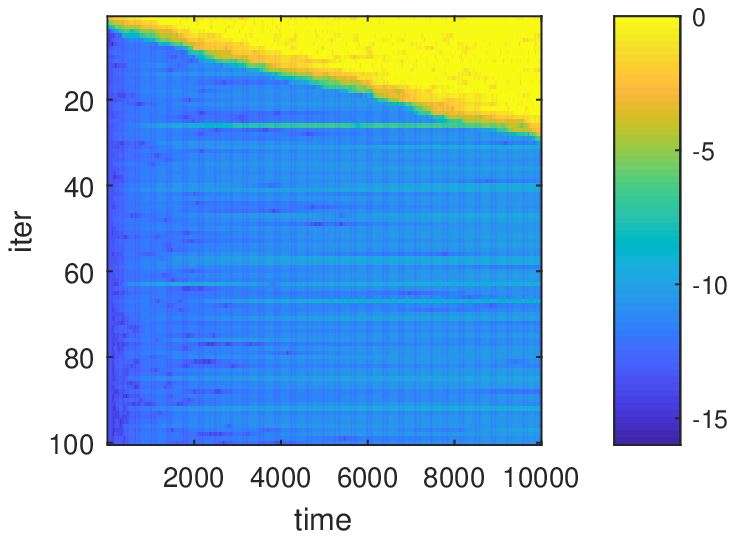}
\par\end{centering}
\caption{\label{fig:spinorbit_error}Log (base 10) errors of standard parareal
(left) and lower dimensional interpolation based $\theta$-parareal
(right) on spin orbit problem, example 3.5. The parameter is $T=10^{4},\epsilon=0.01,\alpha=10^{-4},\phi=0.2$.
Velocity Verlet is used for both fine and coarse integrators with
$h=10^{-2}$ and $H=1$. }
\end{figure}

\begin{example}
\label{exa:One-body-problem}The Kepler one-body problem in 2D. Consider,
\[
\left[\begin{array}{c}
q'\\
p'
\end{array}\right]=\left[\begin{array}{c}
p\\
-\dfrac{q}{\|q\|^{3}}
\end{array}\right],\,\,\,q(0)=1-e,\;p(0)=\sqrt{\dfrac{1+e}{1-e}},
\]
where $q(t)$, $p(t)\in\mathbb{R}^{2}$ and $0\leq e<1$ is the eccentricity.
Larger eccentricity corresponds to stiffer problem. 
\end{example}
In Figure \ref{fig:keplerinterp_error} we present a comparison of
the results computed by the standard parareal and by the interpolative
$\theta$-parareal as described in Algorithm~\ref{alg:Interp_based_parareal}.
Both fine and coarse solvers apply Velocity Verlet with step sizes
$h=10^{-4}$ and $H=0.02$. We see that on intermediate time intervals
(up to around $T=100$), the interpolative approach significantly
improves both the accuracy and stability of parareal. At longer times,
accuracy deteriorates for both methods, although slower with $\theta$-parareal.
Sequentializing sets of parareal simulations to integrate shorter
time intervals at a time greatly improves the overall accuracy of
long time scales for this types of nonlinear Hamiltonian dynamics. 

Figure~\ref{fig:Local-error-1-body} shows the local errors, $|F_{H}u_{n}^{(k+1)}-\theta_{n}^{(k)}C_{H}(u_{n}^{(k+1)})|$
with $N=1000/H=50000$ at different iterations.

\begin{figure}
\begin{centering}
\includegraphics{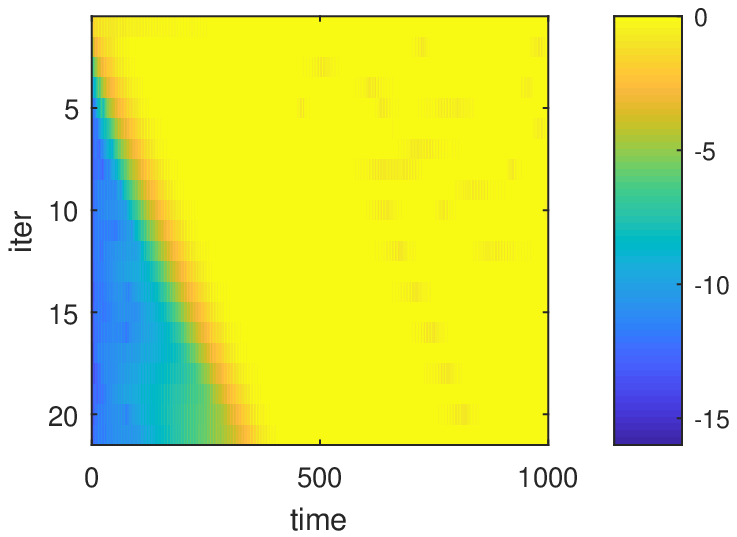}\includegraphics{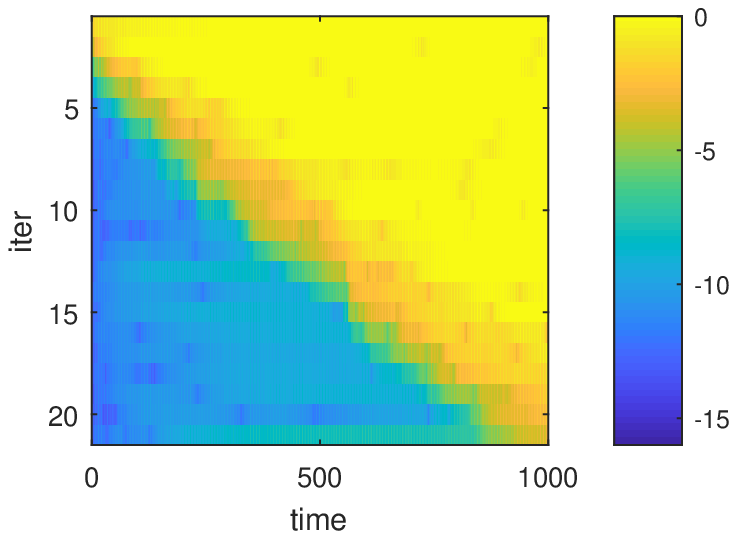}
\par\end{centering}
\begin{centering}
\includegraphics{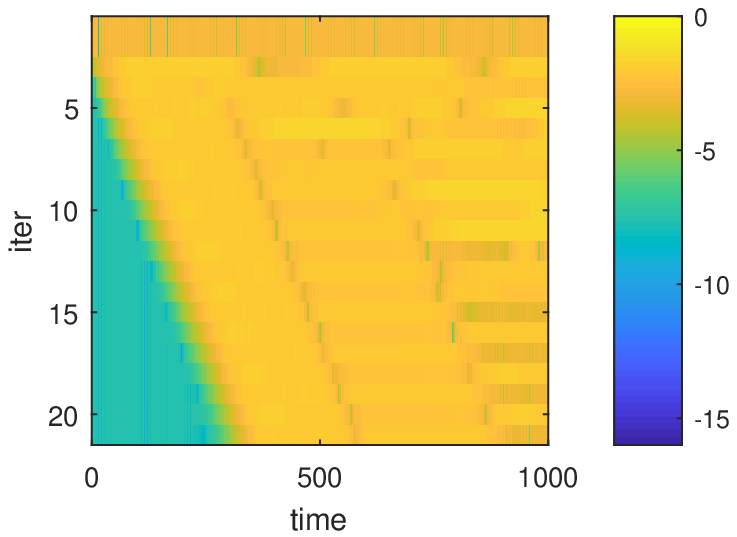}\includegraphics{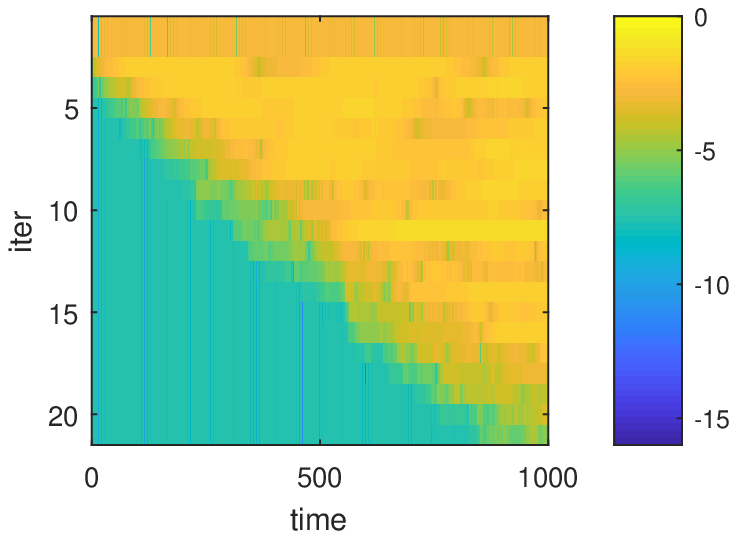}
\par\end{centering}
\caption{\label{fig:keplerinterp_error}Log (base 10) errors of the plain parareal
(left) and interpolative $\theta$-parareal (right) in trajectory
(top) and energy (bottom) for Kepler system of 1 planet in 2D, example
3.6. The errors is presented in $\log_{10}$. Eccentricity is $e=0.5$.
Fine and coarse integrator are Velocity Verlet with steps $h=10^{-4}$
and $H=0.02$. The largest trajectory error is about the diameter
of the orbit.}
\end{figure}
\begin{figure}
\begin{centering}
\includegraphics{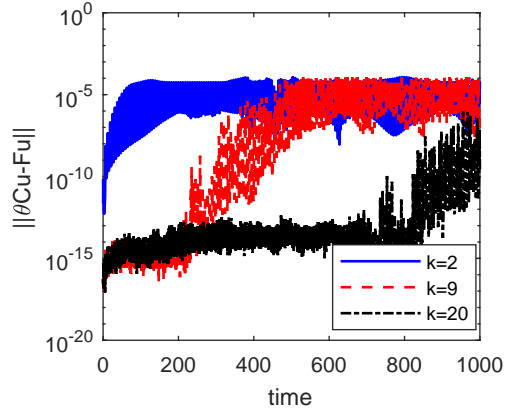}
\par\end{centering}
\caption{\label{fig:Local-error-1-body}Example 3.6: Local errors of the improved
coarse integrator $\theta_{n}^{(k)}C_{H}$. }
\end{figure}

\phantom{}
\begin{example}
\label{exa:The-two-body-Kepler-3D}The two-body Kepler problem in
3 dimensions. The two planets are described by $\mathbb{R}^{3}$,
$j=1,2,$ leading to a nonlinear system in $\mathbb{R}^{12}$.
\[
\left[\begin{array}{c}
q_{1}'\\
q_{2}'\\
p_{1}'\\
p_{2}'
\end{array}\right]=\left[\begin{array}{c}
p_{1}\\
p_{2}\\
-\dfrac{q_{1}}{\|q_{1}\|^{3}}-10^{-5}\dfrac{q_{1}-q_{2}}{\|q_{1}-q_{2}\|^{3}}\\
-\dfrac{q_{2}}{\|q_{2}\|^{3}}+10^{-5}\dfrac{q_{1}-q_{2}}{\|q_{1}-q_{2}\|^{3}}
\end{array}\right],
\]
with the initial conditions is
\[
q_{1}=\left[\begin{array}{c}
1-e_{1}\\
0\\
0
\end{array}\right],\;q_{2}=\left[\begin{array}{c}
\cos(\pi/4)(1-e_{2})\\
0\\
\sin(\pi/4)(1-e_{2})
\end{array}\right],\;p_{1}=\left[\begin{array}{c}
0\\
\sqrt{\dfrac{1+e_{1}}{1-e_{1}}}\\
0
\end{array}\right],\;p_{2}=\left[\begin{array}{c}
0\\
\sqrt{\dfrac{1+e_{2}}{1-e_{2}}}\\
0
\end{array}\right].
\]

In Figure \ref{fig:kepler3D_ReducedDim}, we present a comparison
of the results computed by the standard parareal and by the interpolative
$\theta$-parareal as described in Algorithm~\ref{alg:Interp_based_parareal}.
Both the fine and coarse solvers are Velocity Verlet with step sizes
$h=10^{-4}$ and $H=0.02$. Figure~\ref{fig:interp-dims} shows the
dimensions of the subspaces in which interpolations is computed and
the corresponding computed errors. 

In Figure~\ref{fig:kepler2interp_error3D}, we further compare the
results reported above to the one computed by sequentially applying
the interpolative $\theta$-parareal algorithm in smaller time intervals,
each of which involves $N'=500/H=25000$ steps. We see that smaller
total coarse time steps results in smaller amplification factor, which
consequently results in a significantly improved accuracy. 
\end{example}
\begin{figure}
\begin{centering}
\includegraphics{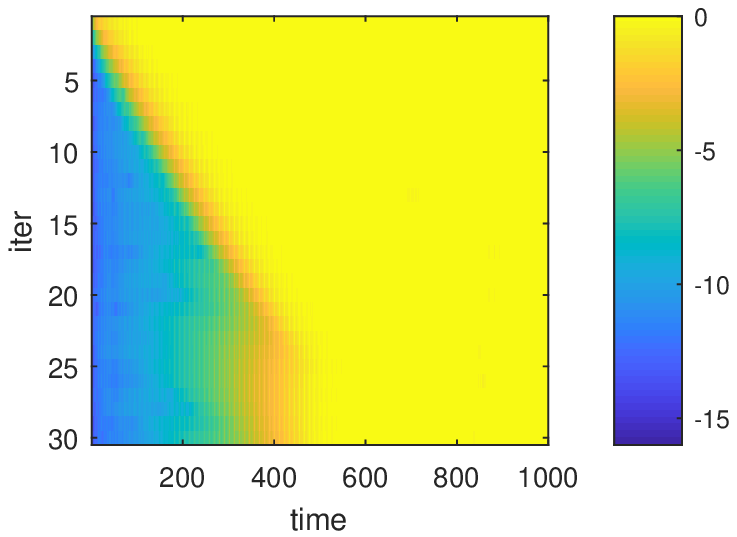}\includegraphics{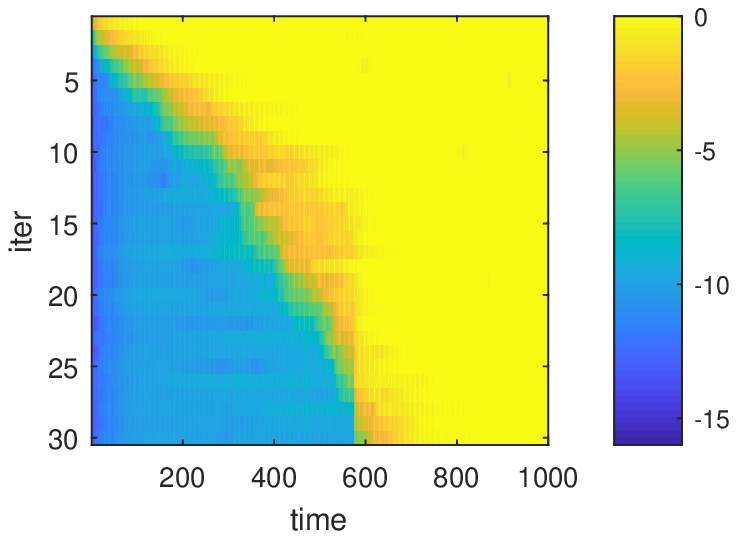}
\par\end{centering}
\begin{centering}
\includegraphics{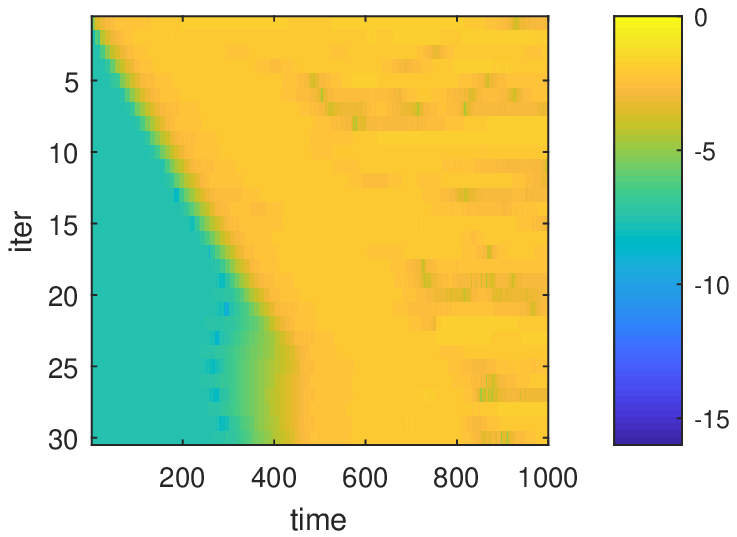}\includegraphics{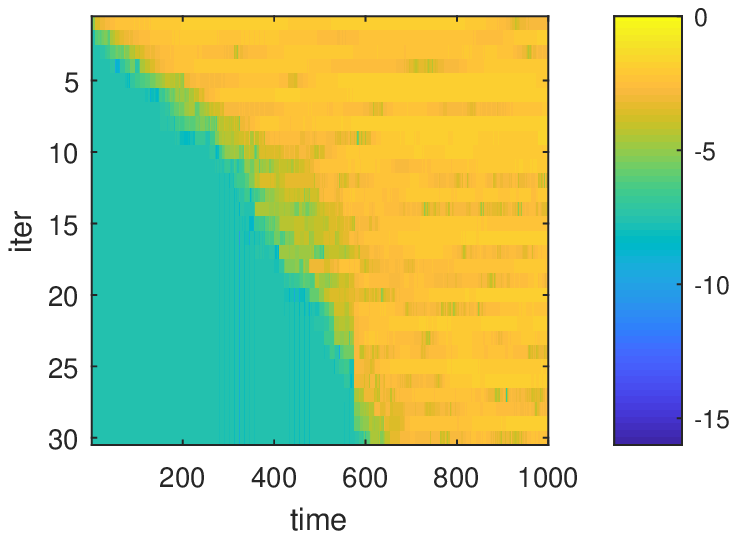}
\par\end{centering}
\caption{\label{fig:kepler3D_ReducedDim}Log (base 10) errors of plain parareal
(left column) and $\theta$-parareal lower dimension (right column).
Errors in trajectory (top row) and energy (bottom row) for a Kepler
system of 2 planets in 3D, example 3.8. Eccentricities are $e_{1,2}=0.4,0.5$.
Fine and coarse integrator are Velocity Verlet with $h=10^{-4}$ and
$H=0.02$. The interaction coefficient between the two masses is $g_{12}=10^{-5}$.}
\end{figure}

\begin{figure}
\begin{centering}
\includegraphics{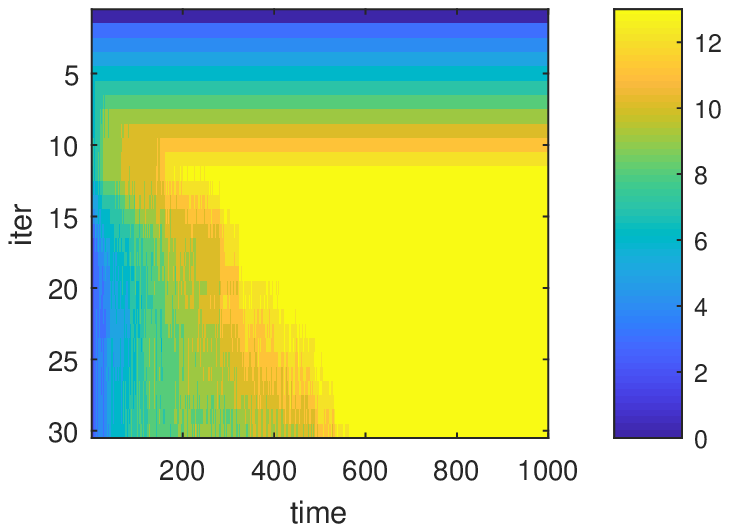}\includegraphics{Kepler2planet3D_ec0405_H002_T1E3_K30_newinterp_TrajError}
\par\end{centering}
\begin{centering}
\includegraphics{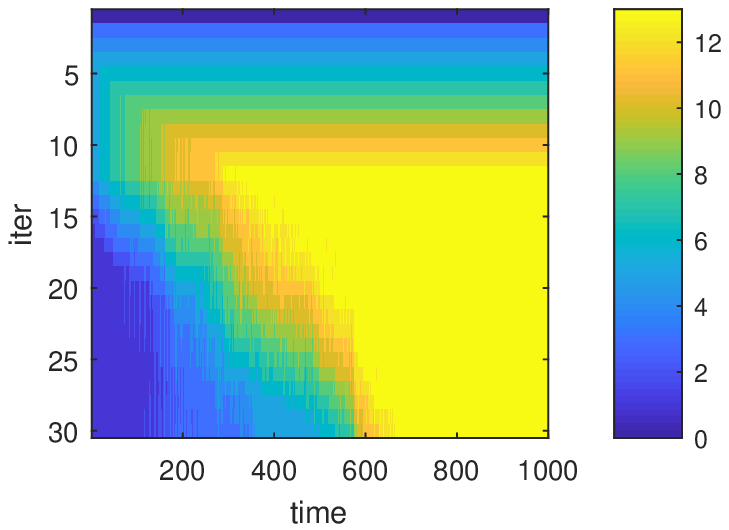}\includegraphics{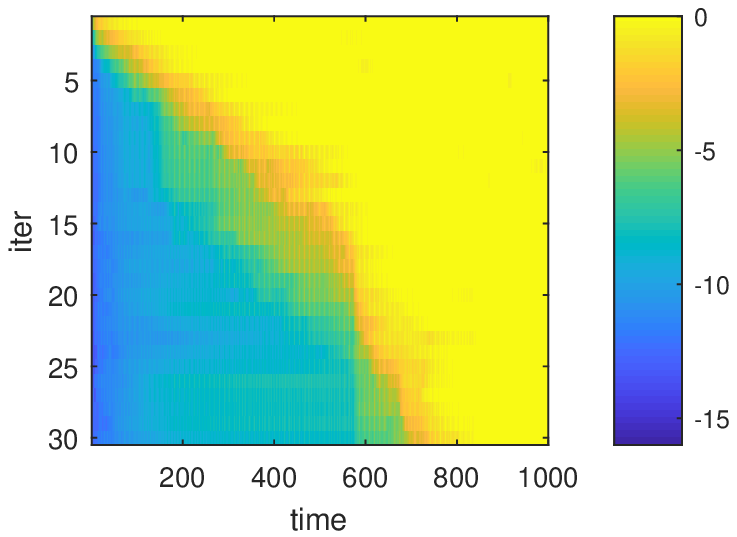}
\par\end{centering}
\begin{centering}
\includegraphics{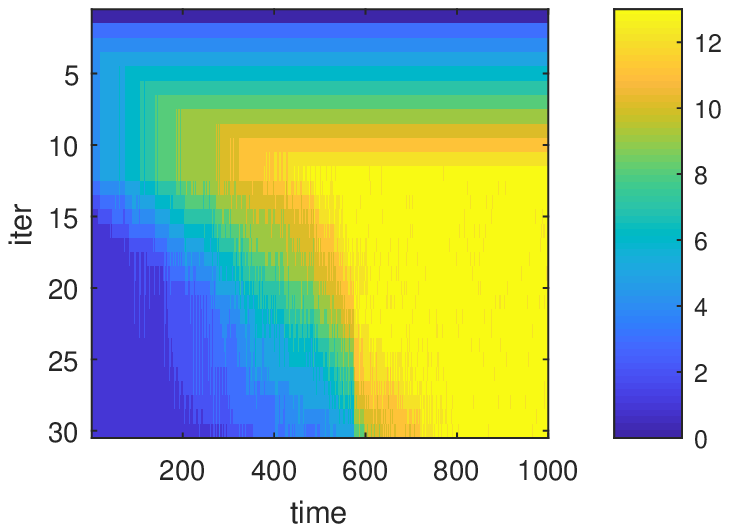}\includegraphics{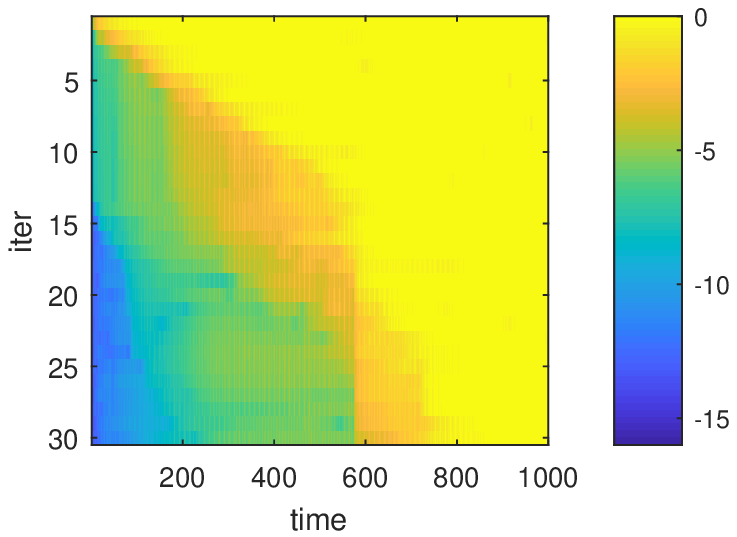}
\par\end{centering}
\caption{\label{fig:interp-dims}Number of singular values used in the subspace
interpolation parareal (left column) and its corresponding errors
(right column) for the two body Kepler problem in 3D, example 3.8.
The tolerance parameter $tol$ in Algorithm~\ref{alg:Interp_based_parareal}
is set to $10^{-14}$ (top row), $10^{-10}$ (middle row), $10^{-6}$
(bottom row). }
\end{figure}

\begin{figure}
\begin{centering}
\includegraphics{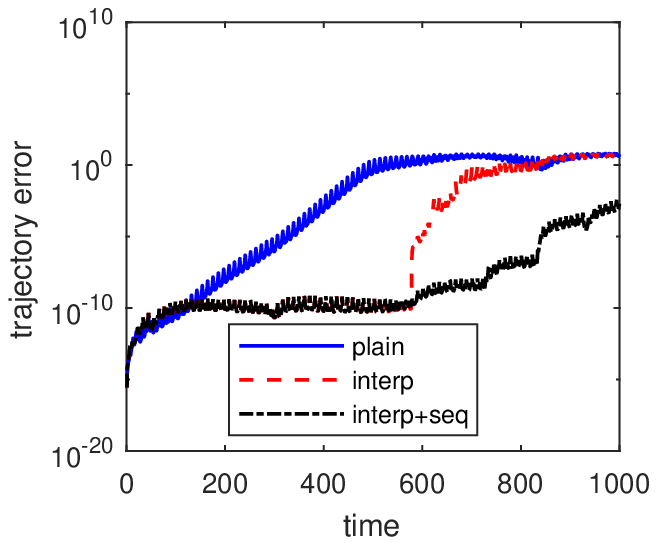}\includegraphics{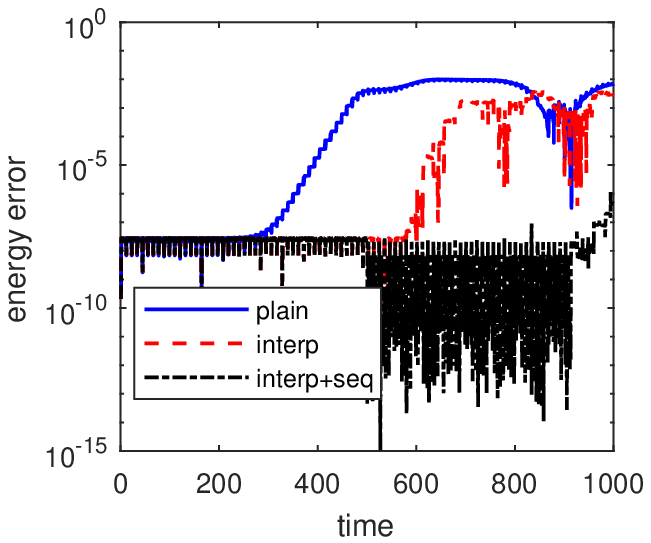}
\par\end{centering}
\caption{\label{fig:kepler2interp_error3D}Error of lower dimensional interpolation
$\theta$-parareal with sequential approach comparing to standard
parareal on Kepler system of 2 planets in 3D, example 3.8. After $K=30$
iterations. For sequential parareal, the cut-off time is at $t=500$. }
\end{figure}

\section{Summary}

In this paper, we proposed a class of parallel-in-time numerical integrators,
built on top of the framework of the parareal schemes. The integrators
are conceived specifically with the objectives of allowing stable
coupling between the chosen coarse- and fine- integrators, which may
be consistent with similar but different differential equations. 

The stability of parareal iterations is largely determined by the
sums of amplification factor of the coarse integrator. While for dissipative
problems, the amplification factor of typical schemes will be strictly
less than one, purely oscillatory problems tend to preserve certain
invariances and thus the amplification factor has modulus $1$. In
this case, we have shown that the $\theta$-parareal method we suggest
may enjoy favorable stability properties compared to the standard
$\theta=1$ case.

The simplest form of the proposed scheme is a small modification to
the original parareal scheme, obtained by multiplying the coarse integrator
by a constant. Hamiltonian or high-dimensional systems require more
complicated methods, for example, using interpolation of data points
obtained in previous iterations. We analyzed the convergence of such
approaches, and presented numerical simulations that enjoyed a few
more digits in accuracy when compared to the results computed by the
standard parareal algorithms. 

\section*{Acknowledgments}

Tsai is supported partially by NSF DMS-1620396 and ARO Grant No. W911NF-12-1-0519.
Nguyen is supported by an ICES NIMS fellowship. Tsai also thanks National
Center for Theoretical Sciences Taiwan for hosting his visits where
part of this research was conducted. 

\bibliographystyle{plabbrv}
\bibliography{../../Paper_Submission_1/Paper/parareal_literatures}

\end{document}